\documentclass[12pt,a4paper,reqno]{amsart} %Fr o m detta
\usepackage{footnote}
\usepackage{amssymb}
\usepackage{latexsym}
\usepackage{amsmath}
\usepackage{mathrsfs}
\usepackage[X2,T1]{fontenc}
\usepackage[applemac]{inputenc}
\usepackage{cite}
\usepackage{color}
\usepackage{calc}                   %Detta beh?vs f?r att
\newcommand{\scal}[2]{\langle #1,#2\rangle}
\newcommand{\rr}[1]{\mathbf R^{#1}}

\newcommand{\cc}[1]{\mathbf C^{#1}}
\newcommand{\nm}[2]{\Vert #1\Vert _{#2}}

\newcommand{\sets}[2]{\{ \, #1\, ;\, #2\, \} }

\newcommand{\ep}{\varepsilon}

\newcommand{\cdo}{\, \cdot \, }

\newcommand{\eabs}[1]{\langle #1\rangle}     %%%%%   for <x>

\newcommand{\vrum}{\vspace{0.1cm}}

\newcommand{\IM}{\operatorname{Im}}

\newcommand{\im}{i}

\newcommand{\nn}[1]{{\mathbf N}^{#1}}
\newcommand{\maclA}{\mathcal A}

\newcommand{\maclH}{\mathcal H}
\newcommand{\maclE}{\mathcal E}
\newcommand{\maclD}{\mathcal D}

\newcommand{\maclR}{\mathcal R}
\newcommand{\maclS}{\mathcal S}

\newcommand{\mascE}{\mathscr E}

\newcommand{\mascS}{\mathscr S}

\numberwithin{equation}{section}          %Detta g?r att man f?r
\newtheorem{thm}{Theorem}
\numberwithin{thm}{section}

\newcommand{\rubrik}{}
\newtheorem{prop}[thm]{Proposition}

\newtheorem{lemma}[thm]{Lemma}

\theoremstyle{definition}

\newtheorem{defn}[thm]{Definition}

\theoremstyle{remark}
\newtheorem{rem}[thm]{Remark}

\author{Elmira Nabizadeh}

\address{Department of Mathematics,
Linn{\ae}us University, V{\"a}xj{\"o}, Sweden}

\email{elmira.nabizadehmorsalfard@lnu.se}

\author{Christine Pfeuffer}

\address{Department of Mathematics,
University of Regensburg, Regensburg, Germany}

\email{christine.pfeuffer@mathematik.uni-regensburg.de}

\author{Joachim Toft}

\address{Department of Mathematics,
Linn{\ae}us University, V{\"a}xj{\"o}, Sweden}

\email{joachim.toft@lnu.se}

\title{Paley-Wiener properties for spaces of entire functions}

\keywords{Bargmann transform}

\subjclass[2010]{primary 46F05; 32A25; 32A36;
secondary 35Q40; 30Gxx}

\begin{document}

%\begin{savenotes}

\begin{abstract}
We deduce Paley-Wiener results in the Bargmann setting. At the
same time we deduce characterisations of Pilipovi{\'c} spaces of
low orders. In particular we improve the characterisation of the
Gr{\"o}chenig test function space $\maclH _{\flat _1}=\maclS _C$,
deduced in \cite{Toft15}.
\end{abstract}

\maketitle

\par

%%%%%%%%%%%%%%%%%%%%%%%
\section{Introduction}\label{sec0}
%%%%%%%%%%%%%%%%%%%%%%%

\par

Paley-Wiener theorems characterize functions and distributions
with certain restricted supports in terms of estimates of their 
Fourier-Laplace transforms. For example, let $f$ be a distribution
on $\rr d$ and let $B_{r_0}(0)\subseteq \rr d$ be the ball with center
at origin and radius $r_0$. Then $f$ is supported in $B_{r_0}(0)$ if 
and only if 
$$
|\widehat{f}(\zeta)| \lesssim \eabs\zeta ^{N} e^{r_0 |\IM (\zeta )|}, 
\quad \zeta\in \cc d,
$$
for some $N\ge 0$. Furthermore, $f$ is supported in $B_{r_0}(0)$ 
and smooth, if and only if
$$
|\widehat{f}(\zeta)| \lesssim \eabs\zeta ^{-N} e^{r_0 |\IM (\zeta )|}, 
\quad \zeta\in \cc d,
$$
for every $N\ge 0$ (see e.{\,}g. \cite[Section 7]{Ho1}).

\par

A similar approach for ultra-regular functions of Gevrey types and 
corresponding ultra-distribution spaces can be done. In fact, let
$s>1$, $\maclD_s'(\rr d)$ be the set of all Gevrey distributions of
order $s$ and let $\maclE_s(\rr d)$ be the set of all smooth functions 
with Gevrey regularity $s$. (See \cite{Ho1} and Section \ref{sec1}
for notations.) Then it can be proved that $f\in \maclD_s'(\rr d)$
is supported in $B_{r_0}(0)$, if and only if 
$$
|\widehat{f}(\zeta)| \lesssim  e^{r_0 |\IM (\zeta )|+ r |\zeta|^{\frac{1}{s}}}, 
\quad \zeta\in \cc d,
$$
for every $r>0$. Furthermore $f\in \maclE_s(\rr d)$ is supported in
$B_{r_0}(0)$, if and only if 
$$
|\widehat{f}(\zeta)| \lesssim  e^{r_0 |\IM (\zeta )|- r |\zeta|^{\frac{1}{s}}}, 
\quad \zeta\in \cc d,
$$
for some $r>0$.

\par

We observe that $s$ in the latter result can not be pushed to be smaller,
because if $s\leq 1$, it does not make any sense to discuss compact support 
properties of $\maclD_s'(\rr d)$ and $\maclE_s(\rr d)$.

\par

In the paper we consider analogous Paley-Wiener properties when the
Fourier-Laplace transform above is replaced by the reproducing kernel
$\Pi _A$ of the Bargmann transform, and the image spaces are replaced by
suitable subspaces of entire functions on $\cc d$.
These subspaces were considered in \cite{Toft15,FeGaTo2} and are given by
\begin{equation}\label{Eq:IntrMainSpaces}
\begin{alignedat}{2}
\maclA _{\flat _\sigma} (\cc d)
&=
\bigcup _{r>0} \maclA _{r,\flat _\sigma} (\cc d),&
\qquad 
\maclA _{0,\flat _\sigma} (\cc d)
&=
\bigcap _{r>0} \maclA _{r,\flat _\sigma} (\cc d),
\\[1ex]
\maclA _{s} (\cc d)
&=
\bigcup _{r>0} \maclA _{r,s} (\cc d),&
\qquad 
\maclA _{0,s} (\cc d)
&=
\bigcap _{r>0} \maclA _{r,s} (\cc d),
\end{alignedat}
\end{equation}
when $\sigma >0$ and $0<s<\frac 12$,
where $ \maclA _{r,\flat _\sigma} (\cc d)$ and $ \maclA _{r,s} (\cc d)$
are the sets of all entire $F$ on $\cc d$ such that
$$
|F(z)|\lesssim e^{r|z|^{\frac {2\sigma}{\sigma +1}}}
\quad \text{respective}\quad
|F(z)|\lesssim e^{r(\log \eabs z)^{\frac 1{1-2s}}}.
$$
The spaces in \eqref{Eq:IntrMainSpaces} appear naturally when
considering the Bargmann transform images of an extended class
of Fourier invariant Gelfand-Shilov spaces, called Pilipovi{\'c} spaces
(see \cite{Toft15,FeGaTo2}).

\par

If $(z,w)$ is the scalar product of $z,w\in \cc d$, then
the reproducing kernel of the Bargmann transform is given by
\begin{align}
(\Pi _AF)(z) &= \pi ^{-d}\scal {F}{e^{(z,\cdo )-|\cdo |^2}},
\notag
\intertext{when
$F$ is a suitable function or (ultra-)distribution. If}
z&\mapsto F(z) e^{R|z|-|z|^2}
\in L^1(\cc d), \quad R>0
\label{Eq:FWeightL1Cond}
\intertext{holds and $d\lambda (w)$ is the Lebesgue measure
on $\cc d$, then
}
(\Pi _AF)(z) &= \pi ^{-d}\int _{\cc d}F(w)e^{(z,w)-|w|^2}\, d\lambda (w).
\notag
\end{align}

\par

A recent Paley-Wiener result with respect to the transform $\Pi _A$
and to the image spaces \eqref{Eq:IntrMainSpaces} is given in
\cite{Toft15}, where it is proved that if $L^\infty _c(\cc d) = \mascE '(\cc d)\cap
L^\infty (\cc d)$, then
\begin{equation}\label{Eq:OriginalMappings}
\Pi _A(\mascE '(\cc d)) = \Pi _A(L^\infty _c (\cc d)) =
\maclA _{\flat _1}(\cc d).
\end{equation}
Evidently, $L^\infty _c \subseteq \mascE '$, and the gap
between these spaces is rather large. It might therefore
be somehow surprising that the first equality holds in
\eqref{Eq:OriginalMappings}.

\par

In Section \ref{sec2} we improve 
\eqref{Eq:OriginalMappings} in different ways. Firstly we
show that we may replace $L^\infty _c(\cc d)$ in
\eqref{Eq:OriginalMappings} with the smaller space
$L^\infty _{A,c}(\cc d)$ given by
$$
L^\infty _{A,c}(\cc d) = \bigcup L^\infty _{A,c}(K),
$$
where $L^\infty _{A,c}(K)$ is the set of
all $F\cdot \chi _K$, where $F$ is analytic in a
neighbourhood of $K$ and $\chi _K$ is the characteristic
function of $K$. Secondly, we may replace $\mascE '(\cc d)$
with the set $\maclE _s'(\cc d)$ of all compactly supported
Gevrey distributions of order $s>1$. Summing up we improve
\eqref{Eq:OriginalMappings} into
\begin{equation}\tag*{(\ref{Eq:OriginalMappings})$'$}
\Pi _A(\maclE '_s(\cc d)) = \Pi _A(L^\infty _{A,c} (\cc d)) =
\maclA _{\flat _1}(\cc d),\qquad s>1.
\end{equation}

\par

In Section \ref{sec2} we also
deduce various kind of related mapping properties when
$\maclA _{\flat _1}(\cc d)$ in \eqref{Eq:OriginalMappings}
is replaced by any of the spaces in \eqref{Eq:IntrMainSpaces}.
More precisely, let $\chi \in L^\infty _c(\cc d)$ be non-negative,
radial symmetric in each complex variable $z_j$ and bounded
from below by a positive constant near the origin. Then we prove
\begin{alignat*}{2}
\Pi _A(\maclA _{0,\flat _{\sigma _0}}' (\cc d)\cdot \chi )
&=
\maclA _{\flat _{\sigma}} (\cc d),
& \quad \sigma &\in ({\textstyle{\frac 12}},1),\  \sigma _0 = \frac \sigma{2\sigma -1},
\\[1ex]
\Pi _A(\maclA _{\flat _{\sigma _0}} (\cc d)\cdot \chi )
&=
\maclA _{\flat _{\sigma}} (\cc d),
& \quad \sigma &\in (0,{\textstyle{\frac 12}}),\  \sigma _0 = \frac \sigma{1-2\sigma},
\\[1ex]
\Pi _A(\maclA _{s} (\cc d)\cdot \chi )
&=
\maclA _{s} (\cc d),
& \quad s &\in [0,{\textstyle{\frac 12}}).
\end{alignat*}
Some related properties are deduced for $\sigma =\frac 12$, as well as
when $\maclA _{\flat _\sigma}$ and
$\maclA _s$ are replaced by $\maclA _{0,\flat _\sigma}$ and
$\maclA _{0,s}$, respectively. (Cf. Theorems \ref{Thm:MainResult1}--\ref{Thm:MainResult6}
and Propositions \ref{Prop:MainResult1}--\ref{Prop:MainResult6}.)

\par

Finally, in Section \ref{sec3} we use the results in Section \ref{sec2} to
deduce characterizations of Pilipovi{\'c} spaces of small orders.

\par

%%%%%%%%%%%%%%%%%%%%%%%
\section{Preliminaries}\label{sec1}
%%%%%%%%%%%%%%%%%%%%%%%

\par

In this section we recall some basic facts. We start by discussing
Pilipovi{\'c} spaces and some of their properties.
Then we recall some facts on modulation spaces. Finally
we discuss the Bargmann transform and some of its
mapping properties, and introduce suitable classes of entire functions on
$\cc d$.

\par

%%% %
\subsection{The Pilipovi{\'c} spaces}\label{subsec1.2}
%%% %
The definition of Pilipovi{\'c} spaces is based on the Hermite functions,
which are given by
$$
h_\alpha (x) = \pi ^{-\frac d4}(-1)^{|\alpha |}
(2^{|\alpha |}\alpha !)^{-\frac 12}e^{\frac {|x|^2}2}
(\partial ^\alpha e^{-|x|^2}),
\quad \alpha \in \nn d.
$$
%It follows that
%$$
%h_{\alpha}(x)=   ( (2\pi )^{\frac d2} \alpha ! )^{-1}
%e^{-\frac {|x|^2}2}p_{\alpha}(x),
%$$
%for some polynomial $p_\alpha$ on $\rr d$, which is
%called the Hermite polynomial of order $\alpha$. 
The Hermite functions are eigenfunctions for the Fourier transform, and
for the Harmonic oscillator $H_d\equiv |x|^2-\Delta$ which acts on functions
and (ultra-)distributions defined
on $\rr d$. More precisely, we have
$$
H_dh_\alpha = (2|\alpha |+d)h_\alpha .
$$

\par

It is well-known that
the set of Hermite functions is a basis of $\mascS (\rr d)$ 
and an orthonormal basis of $L^2(\rr d)$ (cf. \cite{ReSi}).
In particular, if $f\in L^2(\rr d)$, then
$$
\nm f{L^2(\rr d)}^2 = \sum _{\alpha \in \nn d}|c_h(f,\alpha )|^2,
$$
where
\begin{align}
f(x) &= \sum _{\alpha \in \nn d}c_h(f,\alpha )h_\alpha 
\label{Eq:HermiteExp}
\intertext{is the Hermite seriers expansion of $f$, and}
c_h(f,\alpha ) &= (f,h_\alpha )_{L^2(\rr d)}
\label{Eq:HermiteCoeff}
\end{align}
is the Hermite coefficient of $f$ of order $\alpha \in \rr d$.

\par

In order to define the full
scale of Pilipovi{\'c} spaces, their order $s$ should belong to
the extended set
$$
\mathbf R_\flat = \mathbf R_+\cup \sets {\flat _\sigma}{\sigma \in \mathbf R_+},
$$
of positive real numbers, with extended inequality relations as
$$
s_1<\flat _\sigma <s_2
\quad \text{and}\quad \flat _{\sigma _1}<\flat _{\sigma _2}
$$
when $s_1<\frac 12\le s_2$ and $\sigma _1<\sigma _2$. (Cf. \cite{Toft15}.)

\par

For such $r\in \rr d_+$ and $s\in \mathbf R_\flat$ we set
\begin{align}
\vartheta _{r,s}(\alpha ) &\equiv
\begin{cases}
e^{-(\frac 1{r_1}\cdot \alpha _1^{\frac 1{2s}} +\cdots
+ \frac 1{r_d}\cdot \alpha _d^{\frac 1{2s}})}, &  s\in \mathbf R_+
\setminus \{ \frac 12\},
\\[1ex]
r^{\alpha}\alpha !^{-\frac 1{2\sigma }}, &  s = \flat _\sigma ,
\\[1ex]
r^{\alpha}, &  s =\frac 12 , \qquad \alpha \in \nn d
\end{cases}
\label{varthetarsDef}
\intertext{and}
\vartheta _{r,s}'(\alpha ) &\equiv
\begin{cases}
e^{(\frac 1{r_1}\cdot \alpha _1^{\frac 1{2s}} +\cdots
+ \frac 1{r_d}\cdot \alpha _d^{\frac 1{2s}})}, & s \in \mathbf R_+\setminus \{ \frac 12\},
\\[1ex]
r^\alpha \alpha !^{\frac 1{2\sigma}}, & s=\flat _\sigma ,
\\[1ex]
r^ \alpha , & s = \frac 12  , \qquad \alpha \in \nn d.
\end{cases}
\label{varthetarsDualDef}
\end{align}

\par

\begin{defn}\label{Def:PilSpaces}
Let $s\in \overline{\mathbf R_\flat} = \mathbf R_\flat \cup \{ 0\}$,
and let $\vartheta _{r,s}$ and $\vartheta _{r,s}'$ be as in
\eqref{varthetarsDef} and \eqref{varthetarsDualDef}.
\begin{enumerate}
\item $\maclH _0(\rr d)$ consists of all Hermite polynomials, and
$\maclH _0'(\rr d)$ consists of all formal Hermite series expansions
in \eqref{Eq:HermiteExp};

\vrum

\item if $s\in \mathbf R_\flat$, then $\maclH _s(\rr d)$ ($\maclH _{0,s}(\rr d)$)
consists of all $f\in L^2(\rr d)$ such that
$$
|c_h(f,h_\alpha )| \lesssim \vartheta _{r,s}(\alpha )
$$
holds true for some $r\in \rr d_+$ (for every $r\in \rr d_+$);

\vrum

\item if $s\in \mathbf R_\flat$, then $\maclH _s'(\rr d)$ ($\maclH _{0,s}'(\rr d)$)
consists of all formal Hermite series expansions
in \eqref{Eq:HermiteExp} such that 
$$
|c_h(f,h_\alpha )| \lesssim \vartheta _{r,s}'(\alpha )
$$
holds true for every $r\in \rr d_+$ (for some $r\in \rr d_+$).
\end{enumerate}
The spaces $\maclH _s(\rr d)$ and $\maclH _{0,s}(\rr d)$ are called
\emph{Pilipovi{\'c} spaces of Roumieu respectively Beurling types} of order $s$,
and
$\maclH _s'(\rr d)$ and $\maclH _{0,s}'(\rr d)$ are called
\emph{Pilipovi{\'c} distribution spaces of Roumieu respectively Beurling types}
of order $s$.
\end{defn}

\par

\begin{rem}\label{Remark:GSHermite}
Let $\maclS _s(\rr d)$ and $\Sigma _s(\rr d)$
be the Fourier invariant Gelfand-Shilov spaces of order $s\in \mathbf R_+$
of Roumieu respective Beurling types (see \cite{Toft15} for notations).
Then it is proved in \cite{Pil1,Pil2} that
\begin{alignat*}{2}
\maclH _{0,s}(\rr d) &=\Sigma _s(\rr d)\neq \{ 0\} ,& \quad s&> \frac 12,
\\[1ex]
\maclH _{0,s}(\rr d) &\neq\Sigma _s(\rr d) = \{ 0\} ,& \ s&\le \frac 12,
\\[1ex]
\maclH _s(\rr d) &=\maclS _s(\rr d)\neq \{ 0\} ,& \quad s&\ge \frac 12
\intertext{and}
\maclH _s(\rr d) &\neq \maclS _s(\rr d)= \{ 0\} ,& \quad s&< \frac 12.
\end{alignat*}
\end{rem}

\par

In Proposition \ref{Prop:PilSpaceChar} below we give further
characterisations of Pilipovi{\'c} spaces.

\par

Next we recall the topologies for Pilipovi{\'c} spaces. Let $s\in \mathbf R_\flat$,
$r>0$, and let $\nm f{\maclH _{s;r}}$ and $\nm f{\maclH _{s;r}'}$ be given by
\begin{alignat}{2}
\nm f{\maclH _{s;r}}
&\equiv
\sup _{\alpha \in \nn d} |c_h(f,\alpha )\vartheta _{r,s}'(\alpha )| ,&
\quad 
s&\in \mathbf R_\flat 
\label{Eq:PilSpaceSemiNorm}
%\\[1ex]
%\nm f{\maclH _{s;r}}
%&\equiv
%\sup _{\alpha \in \nn d} |c_h(f,\alpha )r^{|\alpha |}\alpha !^{\frac 1{2\sigma}}|, &
%\quad s &= \flat _\sigma ,\ \sigma >0.
%\label{Eq:PilFlatSpaceSemiNorm}
\intertext{and}
\nm f{\maclH _{s;r}'}
&\equiv
\sup _{\alpha \in \nn d} |c_h(f,\alpha )\vartheta _{r,s}(\alpha )| ,&
\quad 
s&\in \mathbf R_\flat ,
\label{Eq:PilDistSpaceSemiNorm}
%\intertext{and}
%\nm f{\maclH _{s;r}}
%&\equiv
%\sup _{\alpha \in \nn d} |c_h(f,\alpha )r^{|\alpha |}\alpha !^{-\frac 1{2\sigma}}|, &
%\quad s &= \flat _\sigma ,\ \sigma >0.
%\label{Eq:PilDistFlatSpaceSemiNorm}
\end{alignat}
when $f$ is a formal expansion in \eqref{Eq:HermiteExp}.
Then $\maclH _{s;r}(\rr d)$ consists of all expansions \eqref{Eq:HermiteExp}
such that $\nm f{\maclH _{s;r}}$ is finite, and $\maclH _{s;r}'(\rr d)$
consists of all expansions \eqref{Eq:HermiteExp}
such that $\nm f{\maclH _{s;r}'}$ is finite. It follows that both
$\maclH _{s;r}(\rr d)$ and $\maclH _{s;r}'(\rr d)$ are Banach spaces under
the norms $f\mapsto \nm f{\maclH _{s;r}}$ and $f\mapsto \nm f{\maclH _{s;r}'}$,
respectively.

\par

We let the topologies of $\maclH _s(\rr d)$ and $\maclH _{0,s}(\rr d)$
be the inductive respectively projective limit topology of $\maclH _{s;r}(\rr d)$
with respect to $r>0$. In the same way, the topologies of $\maclH _s'(\rr d)$
and $\maclH _{0,s}'(\rr d)$ are
the projective respectively inductive limit topology of $\maclH _{s;r}'(\rr d)$
with respect to $r>0$. It follows that all the spaces in Definition
\ref{Def:PilSpaces} are complete, and that $\maclH _{0,s}(\rr d)$
and $\maclH _s'(\rr d)$ are Fr{\'e}chet spaces with semi-norms
$f\mapsto \nm f{\maclH _{s;r}}$ and $f\mapsto \nm f{\maclH _{s;r}'}$,
respectively.

\par

The following characterisations of Pilipovi{\'c} spaces can be found in \cite{Toft15}.
The proof is therefore omitted.

\par

\begin{prop}\label{Prop:PilSpaceChar}
Let $s\in \mathbf R_+\cup \{ 0\}$ and let $f\in \maclH _0'(\rr d)$.
Then $f\in \maclH _{0,s}(\rr d)$ ($f\in \maclH _s(\rr d)$),
if and only if $f\in C^\infty (\rr d)$ and
satisfies $\nm {H_d^N f}{L^\infty}
\lesssim h^N N!^{2s}$ for every $h>0$ (for some $h>0$).
\end{prop}

\par

From now on we let
\begin{equation}\label{phidef}
\phi (x)=\pi ^{-\frac d4}e^{-\frac {|x|^2}2}.
\end{equation}

\par

%%% %
\subsection{Spaces of entire functions and
the Bargmann transform}\label{subsec1.4}
%%% %
Let $\Omega\subseteq \cc d$ be open. Then $A(\Omega)$ is the set of all
analytic functions in $\Omega$. If instead $\Omega\subseteq \cc d$ is
closed, then $A(\Omega)$ is the set of all functions which are analytic
in an open neighbourhood of $\Omega$. In particular, if $z_0\in \cc d$
is fixed, then $A(\{z_0\} )$ is the set of all complex-valued functions
which are defined and analytic near $z_0$.

\par

We shall now consider the Bargmann transform which is defined by the
formula
\begin{equation*}%\label{bargtransf}
(\mathfrak V_df)(z) =\pi ^{-\frac d4}\int _{\rr d}\exp \Big ( -\frac 12(\scal
z z+|y|^2)+2^{\frac 12}\scal zy \Big )f(y)\, dy,
\end{equation*}
when $f\in L^2(\rr d)$ (cf. \cite{B1}). Here
\begin{gather*}
\scal zw = \sum _{j=1}^dz_jw_j,\quad \text{when} \quad
z=(z_1,\dots ,z_d) \in \cc d
\\[1ex]
\text{and} \quad w=(w_1,\dots ,w_d)\in \cc d,
\end{gather*}
and otherwise $\scal \cdo \cdo $ denotes the duality between test
function spaces and their corresponding duals. It is evident that
$\mathfrak V_df$ is the entire function on $\cc d$, given by
$$
(\mathfrak V_df)(z) =\int _{\rr d}\mathfrak A_d(z,y)f(y)\, dy,
$$
or
\begin{equation}\label{bargdistrform}
(\mathfrak V_df)(z) =\scal f{\mathfrak A_d(z,\cdo )},
\end{equation}
where the Bargmann kernel $\mathfrak A_d$ is given by
$$
\mathfrak A_d(z,y)=\pi ^{-\frac d4} \exp \Big ( -\frac 12(\scal
zz+|y|^2)+2^{\frac 12}\scal zy\Big ).
$$
We note that the right-hand side in \eqref{bargdistrform} makes sense
when $f\in \maclS _{\frac 12}'(\rr d)$ and defines an element in $A(\cc d)$,
since $y\mapsto \mathfrak A_d(z,y)$ can be interpreted as an element
in $\maclS _{\frac 12} (\rr d)$ with values in $A(\cc d)$. Here and in what follows,
$A(\Omega )$ denotes the set of analytic functions on the open set
$\Omega \subseteq \cc d$.

\par

It was proved in \cite{B1} that $f\mapsto \mathfrak V_df$ is a bijective
and isometric map  from $L^2(\rr d)$ to the Hilbert space $A^2(\cc d)
\equiv B^2(\cc d)\cap A(\cc d)$, where $B^2(\cc d)$ consists of all
measurable functions $F$ on $\cc  d$ such that
\begin{equation}\label{A2norm}
\nm F{B^2}\equiv \Big ( \int _{\cc d}|F(z)|^2d\mu (z)  \Big )^{\frac 12}<\infty .
\end{equation}
Here $d\mu (z)=\pi ^{-d} e^{-|z|^2}\, d\lambda (z)$, where $d\lambda (z)$ is the
Lebesgue measure on $\cc d$. We recall that $A^2(\cc d)$ and $B^2(\cc d)$
are Hilbert spaces, where the scalar product is given by
\begin{equation}\label{A2scalar}
(F,G)_{B^2}\equiv  \int _{\cc d} F(z)\overline {G(z)}\, d\mu (z),
\quad F,G\in B^2(\cc d).
\end{equation}
If $F,G\in A^2(\cc d)$, then we set $\nm F{A^2}=\nm F{B^2}$
and $(F,G)_{A^2}=(F,G)_{B^2}$.

\par

Furthermore, Bargmann showed that there is a convenient reproducing
formula on $A^2(\cc d)$. More precisely, let
\begin{equation}\label{reproducing}
(\Pi _AF)(z) \equiv \int _{\cc d}F(w)e^{(z,w)}\, d\mu (w),
\end{equation}
when $z\mapsto F(z)e^{R|z|-|z|^2}$ belongs to $L^1(\cc d)$
for every $R\ge 0$. Here
\begin{gather*}
(z,w) = \sum _{j=1}^dz_j\overline{w_j},\quad \text{when} \quad
z=(z_1,\dots ,z_d) \in \cc d
\\[1ex]
\text{and} \quad w=(w_1,\dots ,w_d)\in \cc d,
\end{gather*}
is the scalar product of $z\in \cc d$ and $w\in \cc d$.
Then it is proved in \cite{B1,B2} that $\Pi _A$ is the orthogonal
projection of $B^2(\cc d)$ onto $A^2(\cc d)$. In particular,
$\Pi _AF =F$ when $F\in A^2(\cc d)$.

\medspace

In \cite{B1} it is also proved that
\begin{equation}\label{BargmannHermite}
\mathfrak V_dh_\alpha  = e_\alpha ,\quad \text{where}\quad
e_\alpha (z)\equiv \frac {z^\alpha}{\sqrt {\alpha !}},\quad z\in \cc d .
\end{equation}
In particular, the Bargmann transform maps the orthonormal basis
$\{ h_\alpha \}_{\alpha \in \nn d}$ of $L^2(\rr d)$ bijectively into the
orthonormal basis $\{ e_\alpha \}_{\alpha \in \nn d}$ of monomials
of $A^2(\cc d)$. Hence, there is a natural way to identify formal
Hermite series expansion by formal power series expansions
\begin{equation}\label{Eq:PowerSeriesExp}
F(z) = \sum _{\alpha \in \nn d}c(F,\alpha )e_\alpha (z),
\end{equation}
by letting the series \eqref{Eq:HermiteExp} be mapped into 
\begin{equation}\label{Eq:PowerSeriesExp2}
\sum _{\alpha \in \nn d}c_h(f,\alpha )e_\alpha (z).
\end{equation}
It follows that if $f,g\in L^2(\rr d)$ and $F,G\in A^2(\cc d)$, then
\begin{equation}\label{Scalarproducts}
\begin{aligned}
(f,g)_{L^2(\rr d)} &= \sum _{\alpha \in \nn d}c_h(f,\alpha ) \overline {c_h(g,\alpha )},
\\[1ex]
(F,G)_{A^2(\cc d)} &= \sum _{\alpha \in \nn d}c(F,\alpha ) \overline {c(G,\alpha )}.
\end{aligned}
\end{equation}
Here and in what follows, $(\cdo ,\cdo )_{L^2(\rr d)}$ and
$(\cdo ,\cdo )_{A^2(\cc d)}$ denote the scalar products in
$L^2(\rr d)$ and $A^2(\cc d)$, respectively. 
Furthermore,
\begin{equation}\label{ScalarproductsRel}
c_h(f,\alpha ) = c(F,\alpha )
\quad \text{when}\quad
F=\mathfrak V_df .
\end{equation}

\par

We now recall the following spaces of power series expansions
given in \cite{Toft15}.

\par

\begin{defn}\label{Def:PowerSeriesSpaces}
Let $s\in \overline{\mathbf R_\flat} = \mathbf R_\flat \cup \{ 0\}$,
and let $\vartheta _{r,s}$ and $\vartheta _{r,s}'$ be as in
\eqref{varthetarsDef} and \eqref{varthetarsDualDef}.
\begin{enumerate}
\item $\maclA _0(\cc d)$ consists of all analytic polynomials on $\cc d$,
and $\maclA _0'(\cc d)$ consists of all formal power series expansions
on $\cc d$ in \eqref{Eq:PowerSeriesExp};

\vrum

\item if $s\in \mathbf R_\flat$, then $\maclA _s(\cc d)$ ($\maclA _{0,s}(\cc d)$)
consists of all $F\in L^2(\cc d)$ such that
$$
|c(F,h_\alpha )| \lesssim \vartheta _{r,s}(\alpha )
$$
holds true for some $r>0$ (for every $r>0$);

\vrum

\item if $s\in \mathbf R_\flat$, then $\maclA _s'(\cc d)$ ($\maclA _{0,s}'(\cc d)$)
consists of all formal power series expansions
in \eqref{Eq:PowerSeriesExp} such that 
$$
|c(F,h_\alpha )| \lesssim \vartheta _{r,s}'(\alpha )
$$
holds true for every $r>0$ (for some $r>0$).
\end{enumerate}
%The spaces $\maclA _s(\cc d)$ and $\maclA _{0,s}(\cc d)$ are called
%\emph{Pilipovi{\'c} spaces of Roumieu respectively Beurling types} of order $s$,
%and
%$\maclA _s'(\cc d)$ and $\maclA _{0,s}'(\cc d)$ are called
%\emph{Pilipovi{\'c} distribution spaces of Roumieu respectively Beurling types}
%of order $s$.
\end{defn}

\par

Let $f\in \maclH _0'(\rr d)$ with formal Hermite series expansion
\eqref{Eq:HermiteExp}. Then the Bargmann transform $\mathfrak V_df$
of $f$ is defined to be the formal
power series expansion \eqref{Eq:PowerSeriesExp2}. It follows that
$\mathfrak V_d$ agrees with the earlier definition when acting on $L^2(\rr d)$,
that $\mathfrak V_d$ is linear and bijective from $\maclH _0'(\rr d)$ to
$\maclA _0'(\cc d)$, and restricts to bijections from the spaces
\begin{equation}\label{Eq:PilSpaces}
\maclH _{0,s}(\rr d),\quad \maclH _s(\rr d),\quad \maclH _s'(\rr d)
\quad \text{and}\quad
\maclH _{0,s}'(\rr d)
\end{equation}
to
\begin{equation}\label{Eq:PowerSeriesSpaces}
\maclA _{0,s}(\cc d),\quad \maclA _s(\cc d),\quad \maclA _s'(\cc d)
\quad \text{and}\quad
\maclA _{0,s}'(\cc d)
\end{equation}
respectively, when $s\in \mathbf R_\flat$. We also let the topologies of the spaces
in \eqref{Eq:PowerSeriesSpaces} be inherited from the spaces in
\eqref{Eq:PilSpaces}.

\par

If $s\in \overline {\mathbf R_\flat}$, $f\in \maclH _s(\rr d)$, $g\in \maclH _s'(\rr d)$,
$F \in \maclA _s(\cc d)$ and $G \in \maclA _s'(\cc d)$, then $(f,g)_{L^2(\rr d)}$
and $(F,G)_{A^2(\cc d)}$ are defined by the formula \eqref{Scalarproducts}.
It follows that \eqref{ScalarproductsRel} holds for such choices of $f$ and $g$.
Furthermore, the duals of $\maclH _s(\rr d)$
and $\maclA _s(\cc d)$ can be identified with $\maclH _s'(\rr d)$
and $\maclA _s'(\cc d)$, respectively, through the forms in \eqref{Scalarproducts}.
The same holds true with
\begin{alignat*}{5}
&\maclH _{0,s}, &\quad &\maclH _{0,s}' ,& \quad &\maclA _{0,s}, &\quad &\text{and}&\quad
&\maclA _{0,s}'
\intertext{in place of}
&\maclH _s, &\quad &\maclH _s' , &\quad &\maclA _s, &\quad &\text{and}&\quad
&\maclA _s',
\end{alignat*}
respectively, at each occurrence.

\par

In order to identify the spaces of power series expansions above with spaces of
analytic functions, we let
\begin{equation}\label{Eq:MrDef}
\begin{aligned}
M_{1,r,s}(z)
&=
\begin{cases}
r_1(\log \eabs {z_1})^{\frac 1{1-2s}} + \cdots +
r_d(\log \eabs {z_d})^{\frac 1{1-2s}}, &
s<\frac 12 ,
\\[1ex]
r_1|z_1|^{\frac {2\sigma}{\sigma +1}}+\cdots
+r_d|z_d|^{\frac {2\sigma}{\sigma +1}}, &
s=\flat _\sigma ,\ \sigma >0,
\\[1ex]
\frac {|z|^2}2 -(r_1|z_1|^{\frac 1s}+\cdots +r_d|z_d|^{\frac 1s}), &
s\ge \frac 12 ,
\end{cases}
\\[1ex]
M_{1,r,s}^0(z)
&=
\begin{cases}
M_{1,r,s}(z), & s\neq \frac 12,
\\[1ex]
r_1|z_1|^2+\cdots +r_d|z_d|^2, & s=\frac 12,
\end{cases}
\\[1ex]
M_{2,r,s}(z)
&=
\begin{cases}
r_1|z_1|^{\frac {2\sigma}{\sigma -1}}+\cdots
+r_d|z_d|^{\frac {2\sigma}{\sigma -1}}, &
s=\flat _\sigma ,\ \sigma >1,
\\[1ex]
\frac {|z|^2}2 +(r_1|z_1|^{\frac 1s}+\cdots +r_d|z_d|^{\frac 1s}), &
s\ge \frac 12 ,
\end{cases}
\\[1ex]
M_{2,r,s}^0(z)
&=
\begin{cases}
M_{2,r,s}(z), & s\neq \frac 12,
\\[1ex]
r_1|z_1|^2+\cdots +r_d|z_d|^2, & s=\frac 12,
\end{cases}
\end{aligned}
\end{equation}
%%
%
%
% The previous formulae should be removed after inserted on the right places in
% later results. /JT
%
%
%
%
%
%
when $r\in \rr d_+$ and $z\in \cc d$. For conveniency
we set $M_r=M_{1,r,\flat _1}$.

\par

By \cite{Toft15} we have the following. The proof is therefore omitted.

\par

\begin{prop}\label{Prop:AnFuncSpaceChar}
Let $M_{1,r,s}$, $M_{1,r,s}^0$, $M_{2,r,s}$ and $M_{2,r,s}^0$ be as in
\eqref{Eq:MrDef} when $s\in \overline{\mathbf R_\flat}$ and $r\in \rr d_+$.
Then
\begin{alignat*}{2}
\maclA _{0,s}(\cc d)
&=
\sets {F\in A(\cc d)}{Fe^{-M_{1,r,s}^0}\in L^\infty (\cc d)\
\text{for every}\ r\in \rr d_+},&\quad s&>0,
\\[1ex]
\maclA _s(\cc d)
&=
\sets {F\in A(\cc d)}{Fe^{-M_{1,r,s}}\in L^\infty (\cc d)\
\text{for some}\ r\in \rr d_+},&\quad s&>0,
\\[1ex]
\maclA _s'(\cc d)
&=
\sets {F\in A(\cc d)}{Fe^{-M_{2,r,s}}\in L^\infty (\cc d)\
\text{for every}\ r\in \rr d_+},
&\quad s &> \flat _1,
\\[1ex]
\maclA _{0,s}'(\cc d)
&=
\sets {F\in A(\cc d)}{Fe^{-M_{2,r,s}^0}\in L^\infty (\cc d)\
\text{for some}\ r\in \rr d_+},
&\quad s &> \flat _1,
\\[1ex]
\maclA _{\flat _1}'(\cc d) &= A(\cc d)
\quad \text{and}\quad
\maclA _{0,\flat _1}'(\cc d) = A(\{ 0\}). & &
\end{alignat*}
\end{prop}

\par

\medspace

Next we recall the link between the Bargmann transform
and certain types of short-time Fourier transforms.
The short-time Fourier transform of $f\in \maclS _{1/2}'(\rr d)$
with respect to the window
$\phi \in \maclS _{1/2}(\rr d)\setminus 0$
is defined by
$$
V_\phi f (x,\xi )\equiv \scal f{\overline{\phi (\cdo -x)}e^{-i\scal \cdo \xi}}.
$$
We assume from now on that $\phi$ is given by \eqref{phidef},

\par

Let $S$ be the dilation operator, defined by
\begin{equation}\label{Sdef}
(SF)(x,\xi ) = F(2^{-\frac 12}x,-2^{-\frac 12}\xi ),
\end{equation}
when $F\in L^1_{loc}(\rr {2d})$. Then it
follows by straight-forward computations that
\begin{multline}\label{bargstft1}
(\mathfrak{V} _d f)(z)  =  (\mathfrak{V} _df)(x+\im \xi ) 
=  (2\pi )^{\frac d2}e^{\frac 12(|x|^2+|\xi|^2)}e^{-i\scal x\xi}
V_\phi f(2^{\frac 12}x,-2^{\frac 12}\xi )
\\[1ex]
=(2\pi )^{\frac d2}e^{\frac 12(|x|^2+|\xi|^2)}e^{-i\scal x\xi}(S^{-1}(V_\phi f))(x,\xi ),
\end{multline}
or equivalently,
\begin{multline}\label{bargstft2}
V_\phi f(x,\xi )  =  
(2\pi )^{-\frac d2} e^{-\frac 14(|x|^2+|\xi |^2)}e^{-\im \scal x \xi /2}(\mathfrak{V} _df)
(2^{-\frac 12}x,-2^{-\frac 12}\xi)
\\[1ex]
=(2\pi )^{-\frac d2}e^{-i\scal x\xi /2}S(e^{-\frac {|\cdo |^2}2}(\mathfrak{V} _df))(x,\xi ).
\end{multline}
We observe that \eqref{bargstft1} and \eqref{bargstft2}
can be formulated as
\begin{equation*}%\label{BargStftlink}
\mathfrak V_d = U_{\mathfrak V}\circ V_\phi ,\quad \text{and}\quad
U_{\mathfrak V}^{-1} \circ \mathfrak V_d =  V_\phi ,
\end{equation*}
where $U_{\mathfrak V}$ is the linear, continuous and bijective operator on
$\mathscr D'(\rr {2d})\simeq \mathscr D'(\cc d)$, given by
\begin{equation}\label{UVdef}
(U_{\mathfrak V}F)(x,\xi ) = (2\pi )^{\frac d2} e^{\frac 12(|x|^2+|\xi |^2)}e^{-i\scal x\xi}
F(2^{\frac 12}x,-2^{\frac 12}\xi ) .
\end{equation}

\par

Let $D_{d,r}(z_0)$ be the polydisc
$$
\sets {z=(z_1,\dots ,z_d)\in \cc d}{|z_j-z_{0,j}|< r_j,\ j=1,\dots ,d}
$$
with center and radii given by
$$
z_0 =(z_{0,1},\dots ,z_{0,d}) \in \cc d
\quad \text{and}\quad
r=(r_1,\dots ,r_d)\in \rr d_+.
$$
Then
$$
A(\cc d) = \cap _{r\in \rr d_+} A(D_{d,r}(z)),\qquad
A(\{0\})  =\cup _{r\in \rr d_+} A(D_{d,r}(z_0)).
$$

\par

%%% %
\subsection{Hilbert spaces of power series expansions
and analytic functions}
%%% %

\par

The spaces in Definition \ref{Def:PowerSeriesSpaces} can also be dscribed
by related unions and intersections of Hilbert spaces of analytic functions and
power series expansions as follows. (See also \cite{Toft15}.)

\par

A weight $\omega$ on a Borel set $\Omega$ is a positive function in
$L^\infty _{loc}(\Omega )$ such that $1/\omega$ belongs to
$L^\infty _{loc}(\Omega )$. Let $\vartheta$ be a weight on $\nn d$,
$\omega$ be a weight on $\cc d$, and let
\begin{align}
\nm F{\maclA _{[\vartheta ]}^2(\cc d)}
&\equiv 
\left (
\sum _{\alpha \in \nn d}
|c(F,\alpha )\vartheta (\alpha )|^2
\right )^{\frac 12},
\intertext{when $F \in \maclA _0'(\cc d)$ is given by \eqref{Eq:PowerSeriesExp}, and}
\nm F{A^2_{(\omega )}(\cc d) }
&\equiv
\left (
\int _{\cc d}|F(z)\omega (2^{\frac 12}\overline z)|^2\, d\mu (z)
\right )^{\frac 12},
\end{align}
when $F\in A(\cc d)$. We let $\maclA _{[\vartheta ]}^2(\cc d)$ be the set
of all $F\in \maclA _0'(\cc d)$ such that
$\nm F{\maclA _{[\vartheta ]}^2}$ is finite, and $A^2_{(\omega )}(\cc d)$
be the set of all $F\in A(\cc d)$ such that $\nm F{A^2_{(\omega )}}$ is finite.
It follows that these spaces are Hilbert spaces under these norms.

\par

If $\vartheta$ and $\omega$ are related to each others as
\begin{alignat}{2}
\vartheta (\alpha )
&=
\left (  
\frac {1}{\alpha !} \int _{\mathbf R_+^d} \omega _0(r)^2r^\alpha
\, dr  
\right )
^{\frac 12} ,& \quad \alpha &\in \nn d 
\label{omegavarthetaRel}
\intertext{and}
\omega (z)
&=
e^{\frac {|z|^2}2}\omega _0(|z_1|^2,\dots ,|z_d|^2), & \quad z\in \cc d,
\label{omega0omegaRel}
\end{alignat}
for some suitable weight $\omega _0$ on $\mathbf R_+^d$, then
the following multi-dimensional version of \cite[Theorem (4.1)]{JaEi}
shows that $\maclA _{[\vartheta ]}^2(\cc d) 
= A^2_{(\omega )}(\cc d)$ with equal norms. Here we identify
entire functions with their power series expansions at origin.
%Consequently, the Bargmann transform is bijective and isometric from
%$\maclH _{[\vartheta ]}^2(\rr d)$ to $A^2_{(\omega )}(\cc d)$ for such
%choices of $\vartheta$ and $\omega$.

\par

\begin{thm}\label{cA2sCharExt}
Let $e_\alpha$ be as in \eqref{BargmannHermite}, $\alpha \in \nn d$,
and let $\omega _0$ be a positive measurable function on $\mathbf R_+^d$.
Also let $\vartheta$ and $\omega$ be weights on $\nn d$ and $\cc d$,
respectively, related to each others by \eqref{omegavarthetaRel} and
\eqref{omega0omegaRel}, and such that
\begin{equation}\label{varthetaCond}
\frac {r^{|\alpha |}}{(\alpha !)^{\frac 12}}\lesssim \vartheta (\alpha ),\quad
\alpha \in \nn d,
\end{equation}
holds for every $r>0$.
Then $\maclA ^2_{[\vartheta ]}(\cc d)=A^2_{(\omega )}(\cc d)$
with equality in norms. 
\end{thm}

\par

We omit the proof of Theorem \ref{cA2sCharExt},
since the result is an immediate consequence of
\cite[Theorem 3.5]{Toft15}.

\par

In our situation, the involved weights should satisfy a
split condition. In one
dimension, \eqref{omegavarthetaRel}, \eqref{omega0omegaRel} and
\eqref{varthetaCond} take the forms
\begin{alignat}{2}
\vartheta _j(\alpha _j)
&=
\left ( 
\frac {1}{\alpha _j!} \int _{\mathbf R_+} \omega _{0,j}(r)^2r^{\alpha _j}
\, dr  
\right )
^{\frac 12},&\quad \alpha _j&\in \mathbf N,\tag*{(\ref{omegavarthetaRel})$'$}
\\[1ex]
\omega _j(z_j)
&=
e^{\frac {|z_j|^2}2}\omega _{0,j}(|z_j|^2),&\quad z_j &\in \mathbf C
\tag*{(\ref{omega0omegaRel})$'$}
\intertext{and}
\frac {r^{|\alpha _j|}}{(\alpha _j!)^{\frac 12}} &\lesssim \vartheta _j(\alpha _j),&
\quad r>0,\ \alpha _j &\in \mathbf N.
\tag*{(\ref{varthetaCond})$'$}
\end{alignat}

\par

\begin{lemma}\label{Lemma:SplitWeights}
For $j=1,\dots ,d$, let $\omega _{0,j}$ be weights on $\overline{\mathbf R_+}$,
$\omega _j$ be weights on $\mathbf C$ and $\vartheta _j$ be weights on
$\mathbf N$ such that
\eqref{omegavarthetaRel}$'$--\eqref{varthetaCond}$'$ hold, $j=1,\dots ,d$,
and set $\omega _0(z)\equiv \prod_{j=1}^d \omega _{0,j}(z_j)$,
$z\equiv (z_1, \ldots, z_d) \in \cc d$. If $\vartheta$ and $\omega$ are given by
\eqref{omegavarthetaRel}, and \eqref{omega0omegaRel}, then
\begin{alignat}{2}
\vartheta (\alpha) &= \prod_{j=1}^d
\vartheta_j(\alpha _j),&
\qquad
\alpha &= (\alpha _1, \dots, \alpha _d) \in \nn d
\label{Eq:TensorWeights1}
\intertext{and}
\omega (z) &= \prod_{j=1}^d \omega _j(z_j),&
\qquad
z &= (z_1,\dots ,z_d)\in \cc d.
\label{Eq:TensorWeights2}
    \end{alignat}
\end{lemma}

\par

\begin{proof}
	By \cite[Theorem 3.5]{Toft15} and its proof, it follows that
	$\omega_{0,j}\cdot r^{\alpha _j}
    \in L^1(\mathbf{R}_+)$ for all $j \in \{1, \ldots, d\}$
    and $\alpha_j \in \mathbf{N}$. Hence, Fubini's theorem
    gives
    \begin{multline*}
    	\vartheta (\alpha) =
         \left( 
         \frac{1}{\alpha!} \int_{\mathbf{R}^d_+}
         \omega _0(r)^2 r^{\alpha}\, dr 
         \right)
         ^{\frac 12}
%        \\[1ex]
         = \left( 
         \frac{1}{\alpha!}\prod _{j=1}^d
         \int_0^{\infty}
         \omega_{0,j}(r_j)^2
         r_j^{\alpha _j}\, dr_j
         \right)
         ^{\frac 12}
       \\[1ex]
         =  \prod_{j=1}^d \vartheta_{j}(\alpha _j),
    \end{multline*}
and \eqref{Eq:TensorWeights1} follows. The assertion \eqref{Eq:TensorWeights2}
follows from the definitions.
\end{proof}

\par

%%%%
\subsection{A test function space introduced by Gr{\"o}chenig}
%%%%

\par

In this section we recall some comparison results deduced in \cite{Toft15},
between a test function space, $\maclS _C(\rr d)$, introduced by Gr{\"o}chenig
in \cite{Gc2} to handle modulation spaces with elements in
spaces of ultra-distributions.

\par

The definition of $\maclS _C(\rr d)$ is given as follows. Here we notice that the
continuity of the map $V_\phi$ from $\maclS _{1/2}(\rr d)$ to
$\maclS _{1/2}(\rr {2d})$ implies that its dual, $V_\phi ^*$ from
$\maclS _{1/2}'(\rr {2d})$ to $\maclS _{1/2}'(\rr d)$ is continuous.

\par

\begin{defn}\label{Def:SCSG}
Let $\phi (x)=\pi ^{-\frac d4}e^{-\frac {|x|^2}2}$. Then
$\maclS _C(\rr d)$ and $\maclS _G(\rr d)$ consist of all
$f\in \mascS '(\rr d)$ such that $f=V_\phi ^*F$, for some
$F\in L^\infty (\rr d)\cap \mascE '(\rr d)$ and
$F\in \mascE '(\rr d)$, respectively.
\end{defn}

\par

It follows that $f\in \maclS _C(\rr d)$, if and only if
\begin{equation}\label{Eq:fSTFTAdjDef}
f(x) = (2\pi )^{-\frac d2}\iint _{\rr {2d}} F(y,\eta )e^{-\frac 12|x-y|^2}
e^{i\scal x\eta}\, dyd\eta ,
\end{equation}
for some $F\in L^\infty (\rr d)\cap \mascE '(\rr d)$.

\par

\begin{rem}
By the identity
$(V_\phi h,F)=(h,V_\phi ^*F)$ and the fact that the map
$(f,\phi )\mapsto V_\phi f$ is continuous from $\mascS (\rr d)\times
\mascS (\rr d)$ to $\mascS (\rr {2d})$, it follows that $f=V_\phi ^*F$ is
uniquely defined as an element in $\mascS '(\rr d)$ when
$F\in \mascS '(\rr {2d})$ (cf. \cite{CPRT10}).
In particular, the space $\maclS _G(\rr d)$ in Definition \ref{Def:SCSG}
is well-defined, and it is evident that 
$\maclS _C(\rr d) \subseteq \maclS _G(\rr d)$.
\end{rem}

\par

The following is a restatement of \cite[Lemma 4.9]{Toft15}. The
result is essential when deducing the characterizations of Pilipovi{\'c}
spaces in Section \ref{sec3}.

\par

\begin{lemma}\label{GrochSpaceBargm}
Let $F\in L^\infty (\cc d)\cup \mascE '(\cc d)$. Then the Bargmann transform
of $f=V_\phi ^*F$
is given by $\Pi _AF_0$, where
\begin{equation}\label{Eq:FtoF0}
F_0(x+i\xi ) = (2\pi ^3)^{\frac d4}F(\sqrt 2 x,-\sqrt 2 \xi )
e^{\frac 12(|x|^2+|\xi |^2)}
e^{-i\scal x\xi}, \quad x,\xi \in \rr d.
\end{equation}

\par

Moreover, the images of $\maclS _C(\rr d)$ and $\maclS _G(\rr d)$
under the Bargmann transform are given by
\begin{equation}\label{Eq:PiAGrochImages}
\sets {\Pi _AF}{F\in L^\infty (\cc d)\cap \mascE '(\cc d)}
\quad \text{and}\quad
\sets {\Pi _AF}{F\in \mascE '(\cc d)},
\end{equation}
respectively.
\end{lemma}

\par

% Here recall that the map $\Pi _A$ is defined
% by \eqref{reproducing}.

% \par

The next result is a straight-forward consequence of
\cite[Theorem 4.10]{Toft15}. The proof is therefore omitted.

\par

\begin{prop}\label{Prop:GrochPilipIdentity}
It holds $\maclS _C(\rr d) = \maclS _G(\rr d) =
\maclH _{\flat _1}(\rr d)$.
\end{prop}

\par

Due to the image properties of the Bargmann transform,
the next result is
equivalent with the previous one.

\par

\begin{prop}\label{Prop:GrochPilipBargmIdentity}
The sets in \eqref{Eq:PiAGrochImages} are equal to $\maclA _{\flat _1}(\cc d)$. 
\end{prop}

\par

In the next section we extend Propositions \ref{Prop:GrochPilipIdentity}
and \ref{Prop:GrochPilipBargmIdentity} by proving that the conclusions in
Proposition \ref{Prop:GrochPilipBargmIdentity} hold for suitable smaller and larger
sets than those in \eqref{Eq:PiAGrochImages}. We also deduce similar identifications
for other Pilipovi{\'c} spaces and their Bargmann images.

\par

%%%%%%%%%%%%%%%%%%%%%%%%%%%%%%
\section{Paley-Wiener properties for
Bargmann-Pilipovi{\'c} spaces}\label{sec2}
%%%%%%%%%%%%%%%%%%%%%%%%%%%%%%

\par

In this section we consider spaces of compactly supported functions with interiors in
$\maclA _s(\cc d)$ or in $\maclA _s'(\cc d)$. We show that the images of
such functions under the reproducing kernel $\Pi _A$ are equal to
$\maclA _s(\cc d)$, for some other choice of
$s\le \flat _1$. In the first part we state the main results given in Theorems
\ref{Thm:MainResult1}--\ref{Thm:MainResult6}. They are straight-forward
consequences of Propositions \ref{Prop:MainResult1}--\ref{Prop:MainResult6No2},
which in some sense contain more information. Thereafter we deduce
results which are needed for their proofs. Depending of the choice of $s$,
there are several different situations for characterizing $\maclA _s(\cc d)$.
This gives rise to a quite large flora of main results, where each one takes
care of one situation.

\par

In order to present the main results, we need the following definition.

\par

\begin{defn}\label{Def:RClass}
Let $t_1,t_2\in \rr d_+$ be such that $t_1\le t_2$. Then the function
$\chi \in L^\infty (\cc d)$ is called positive, bounded and radial symmetric
with respect to $t_1$ and $t_2$, if the following conditions are fulfilled:
\begin{itemize}
\item $\chi \in L^\infty (\cc d)\cap \mascE '
(\overline{D_{t_2}(0)})$ is non-negative;

\vrum

\item $\chi (z_1,\dots ,z_d)=\chi _0(|z_1|,\dots ,|z_d|)$
for some function $\chi _0$;

\vrum

\item $\chi \ge c$ on $D_{t_1}(0)$ for some constant $c>0$.
\end{itemize}
The set of positive, bounded and radial symmetric functions
with respect to $t_1$ and $t_2$  is denoted by
$\maclR ^\infty _{t_1,t_2}(\cc d)$, and $\maclR ^\infty (\cc d)$
is defined by
$$
\maclR ^\infty (\cc d) \equiv \bigcup _{t_1\le t_2\in \rr d_+}
\maclR ^\infty _{t_1,t_2}(\cc d).
$$
\end{defn}

\par

\subsection{Main results}

\par

We begin with characterizing the largest spaces in our
investigations, which appears when $s=\flat _1$, and
then proceed with spaces of decreasing order. First we recall
that elements in $\maclA _s(\cc d)$ and $\maclA _{0,s}(\cc d)$
fulfill conditions of the form
\begin{align}
|F(z)| &\lesssim e^{r_1|z_1|^{\frac {2\sigma}{\sigma +1}}+\cdots
+r_d|z_d|^{\frac {2\sigma}{\sigma +1}}},
\label{Eq:PilSpaceCharEst1}
\intertext{when $s=\flat _\sigma$ and of the form}
|F(z)| &\lesssim e^{r_1 (\log \eabs {z_1})^{\frac 1{1-2s}} +\cdots +
r_d (\log \eabs {z_d})^{\frac 1{1-2s}}},
\label{Eq:PilSpaceCharEst2}
\end{align}
when $s\in [0,\frac 12)$.

\par

\begin{thm}\label{Thm:MainResult1}
Let $F\in A(\cc d)$, $\sigma =1$ and $s>1$. Then the following
conditions are equivalent:
\begin{enumerate}
\item $F\in \maclA _{\flat _1}(\cc d)$;

\vrum

\item The estimate \eqref{Eq:PilSpaceCharEst1} holds for some $r\in \rr d_+$;

\vrum

\item For some $r_0\in \rr d_+$ and every $r\in \rr d_+$ with $r_0\le r$
and every $\chi \in \maclR _{r_0,r}^\infty (\cc d)$, there exists $F_0\in
A(\overline{D_r(0)})$ such that
$F = \Pi _A(F_0\cdot \chi )$;

\vrum

\item For some $r_0\in \rr d_+$ and every $r\in \rr d_+$ with $r_0<r$
and some $\chi \in \maclR _{r_0,r}^\infty (\cc d)$, there exists $F_0\in
A(\overline{D_r(0)})$ such that $F = \Pi _A(F_0\cdot \chi )$;

\vrum

\item There exists $F_0\in
\mascE '(\cc d)\cap L^\infty (\cc d)$ such that $F = \Pi _AF_0$;

\vrum

\item There exists $F_0\in
\maclE _s'(\cc d)$ such that $F = \Pi _AF_0$.
\end{enumerate}
\end{thm}

\par

\begin{rem}
Since
$$
\mascE '(\cc d)\cap L^\infty (\cc d) \subseteq \mascE '(\cc d)
\subseteq \maclE _s'(\cc d),
$$
Theorem \ref{Thm:MainResult1} still holds true after $\maclE _s'$
has been replaces by $\mascE '$ in (6).
\end{rem}

\par

\begin{thm}\label{Thm:MainResult2}
Let $F\in A(\cc d)$, $\sigma =1$ and $\chi \in \maclR ^\infty (\cc d)$.
Then the following is true:
\begin{itemize}
\item[(i)]
The following conditions are equivalent:
\begin{enumerate}
\item $F\in \maclA _{0,\flat _1}(\cc d)$;

\vrum

\item The estimate \eqref{Eq:PilSpaceCharEst1} holds for every $r\in \rr d_+$;

\vrum

\item
There exists $F_0\in A(\cc d)$ such that
$F = \Pi _A(F_0\cdot \chi )$;
\end{enumerate}

\vrum

\item[(ii)]
The map $F\mapsto \Pi _A(F\cdot \chi )$ from $A(\cc d)$ to
$\maclA _{0,\flat _1}(\cc d)$ is a homeomorphism.
\end{itemize}
\end{thm}

\par

The next result deals with the case when $s=\flat _\sigma$ with
$\sigma \in (\frac 12,1)$.

\par

\begin{thm}\label{Thm:MainResult3}
Let $F\in A(\cc d)$, $\chi \in \maclR ^\infty (\cc d)$,
$\sigma \in (\frac 12,1)$ and let
$$
\sigma _0 = \frac \sigma{2\sigma -1}. 
$$
Then the following is true:
\begin{itemize}
\item[(i)]
The following conditions are equivalent:
\begin{enumerate}
\item $F\in \maclA _{\flat _\sigma}(\cc d)$ ($F\in
\maclA _{0,\flat _\sigma}(\cc d)$);

\vrum

\item The estimate \eqref{Eq:PilSpaceCharEst1} holds for some
(for every) $r\in \rr d_+$;

\vrum

\item There exists
$F_0\in \maclA _{0,\flat _{\sigma _0}}'(\cc d)$
($F_0\in \maclA _{\flat _{\sigma _0}}'(\cc d)$) such that
$F = \Pi _A(F_0\cdot \chi )$;
\end{enumerate}

\vrum

\item[(ii)]
The mappings $F\mapsto \Pi _A(F\cdot \chi )$ from
$\maclA _{0,\flat _{\sigma _0}}'(\cc d)$ to
$\maclA _{\flat _{\sigma}}(\cc d)$ and
from
$\maclA _{\flat _{\sigma _0}}'(\cc d)$ to
$\maclA _{0,\flat _{\sigma}}(\cc d)$
are homeomorphisms.
\end{itemize}
\end{thm}

\par

The next result deals with the case when $s=\flat _\sigma$ with
$\sigma =\frac 12$.

\par

\begin{thm}\label{Thm:MainResult4}
Let $F\in A(\cc d)$, $\sigma =\frac 12$ and $\chi \in \maclR ^\infty (\cc d)$.
Then the following is true:
\begin{itemize}
\item[(i)]
The following conditions are equivalent:
\begin{enumerate}
\item $F\in \maclA _{\flat _{\sigma}}(\cc d)$
($F\in \maclA _{0,\flat _{\sigma}}(\cc d)$);

\vrum

\item The estimate \eqref{Eq:PilSpaceCharEst1} holds for some (for every)
$r\in \rr d_+$;

\vrum

\item There exists $F_0\in \maclA _{0,1/2}'(\cc d)$
($F_0\in \maclA _{0,1/2}(\cc d)$) such that
$F = \Pi _A(F_0\cdot \chi )$;
\end{enumerate}

\vrum

\item[(ii)]
The mappings $F\mapsto \Pi _A(F\cdot \chi )$ from
$\maclA _{0,1/2}'(\cc d)$ to
$\maclA _{\flat _{1/2}}(\cc d)$ and
from
$\maclA _{0,1/2}(\cc d)$ to
$\maclA _{0,\flat _{1/2}}(\cc d)$
are homeomorphisms.
\end{itemize}
\end{thm}

\par

The next result deals with the case when $s=\flat _\sigma$ with
$\sigma \in (0,\frac 12)$.

\par

\begin{thm}\label{Thm:MainResult5}
Let $F\in A(\cc d)$, $\chi \in \maclR ^\infty (\cc d)$,
$\sigma \in (0,\frac 12)$ and let
$$
\sigma _0 = \frac \sigma{1-2\sigma}. 
$$
Then the following is true:
\begin{itemize}
\item[(i)]
The following conditions are equivalent:
\begin{enumerate}
\item $F\in \maclA _{\flat _\sigma}(\cc d)$ ($F\in
\maclA _{0,\flat _\sigma}(\cc d)$);

\vrum

\item The estimate \eqref{Eq:PilSpaceCharEst1} holds for some
(for every) $r\in \rr d_+$;

\vrum

\item There exists
$F_0\in \maclA _{\flat _{\sigma _0}}(\cc d)$
($F_0\in \maclA _{0,\flat _{\sigma _0}}(\cc d)$) such that
$F = \Pi _A(F_0\cdot \chi )$;
\end{enumerate}

\vrum

\item[(ii)]
The mappings $F\mapsto \Pi _A(F\cdot \chi )$ from
$\maclA _{\flat _{\sigma _0}}(\cc d)$ to
$\maclA _{\flat _{\sigma}}(\cc d)$ and
from
$\maclA _{0,\flat _{\sigma _0}}(\cc d)$ to
$\maclA _{0,\flat _{\sigma}}(\cc d)$
are homeomorphisms.
\end{itemize}
\end{thm}

\par

In the next result we consider the case when $s\in [0,\frac 12)$
is real.

\par

\begin{thm}\label{Thm:MainResult6}
Let $F\in A(\cc d)$, $\chi \in \maclR ^\infty (\cc d)$,
$s \in [0,\frac 12)$.
Then the following is true:
\begin{itemize}
\item[(i)]
The following conditions are equivalent:
\begin{enumerate}
\item $F\in \maclA _s(\cc d)$ ($F\in
\maclA _{0,s}(\cc d)$);

\vrum

\item The estimate \eqref{Eq:PilSpaceCharEst2} holds for some
(for every) $r\in \rr d_+$;

\vrum

\item There exists
$F_0\in \maclA _{s}(\cc d)$
($F_0\in \maclA _{0,s}(\cc d)$) such that
$F = \Pi _A(F_0\cdot \chi )$;
\end{enumerate}

\vrum

\item[(ii)]
The mappings $F\mapsto \Pi _A(F\cdot \chi )$ from
$\maclA _{s}(\cc d)$ to $\maclA _{s}(\cc d)$ and
from $\maclA _{0,s}(\cc d)$ to
$\maclA _{0,s}(\cc d)$
are homeomorphisms.
\end{itemize}
\end{thm}

\par

The previous theorems are essentially consequences of Propositions
\ref{Prop:MainResult1}--\ref{Prop:MainResult6}, where more
detailed information about involved constants are given.

\par

\begin{prop}\label{Prop:MainResult1}
Let $F\in A(\cc d)$, $\sigma =1$, $s>1$ and $r\in \rr  d_+$. Then the
following conditions are equivalent:
\begin{enumerate}
%\item $F\in \maclA _{\flat _1}(\cc d)$;
%
%\vrum
%
\item For some $r_0\in \rr d_+$ such that $r_0<r$, \eqref{Eq:PilSpaceCharEst1}
holds with $r_0$ in place of $r$;

\vrum

\item For some $t_1\in \rr d_+$ such that $t_1<r$, every
$t_2\in \rr d_+$ with $t_1\le t_2 <r$ and every $\chi \in
\maclR _{t_1,t_2}^\infty (\cc d)$, there exists $F_0\in
A(\overline{D_r(0)})$ such that
$F = \Pi _A(F_0\cdot \chi )$;

\vrum

\item For some $t_1\in \rr d_+$ such that $t_1<r$ and every
$t_2\in \rr d_+$ with $t_1\le t_2 <r$, there exist $\chi \in
\maclR _{t_1,t_2}^\infty (\cc d)$ and $F_0\in
A(\overline{D_r(0)})$ such that $F = \Pi _A(F_0\cdot \chi _{D_r(0)})$;

\vrum

\item For some $r_0\in \rr d_+$ such that $r_0<r$,  there exists $F_0\in
\mascE '(D_r(0))\cap L^\infty (\cc d)$ such that $F = \Pi _AF_0$;

\vrum

\item For some $r_0\in \rr d_+$ such that $r_0<r$,  there exists $F_0\in
\maclE '_s(D_{r_0}(0))$ such that $F = \Pi _AF_0$.
\end{enumerate}
\end{prop}

\par

Theorems \ref{Thm:MainResult3}--\ref{Thm:MainResult5} essentially follows from
the following proposition.

\par

\begin{prop}\label{Prop:MainResult345}
Let $\tau >\frac 12$, $r,t_1,t_2\in \rr  d_+$ be such
that $t_1\le t_2$, and let $\chi \in \maclR _{t_1,t_2}^\infty (\cc d)$. Then the
following is true:
\begin{enumerate}
\item Let $F\in A(\cc d)$ be such that
\begin{equation}\label{Eq:MainResult3Eq1}
|F(z)| \lesssim e^{r_{0,1}|z_1|^{\frac {2}{2\tau +1}}+\cdots +r_{0,d}|z_d|
^{\frac {2}{2\tau +1}}}
\end{equation}
holds for some $r_0\in \rr d_+$ such that $r_0<r$. Then for some
$r_0\in \rr d_+$ such that $r_0<r$, there exists $F_0\in A(\cc d)$
such that $F=\Pi _A(F_0\cdot \chi )$ and
\begin{equation}\label{Eq:MainResult3Eq2}
|F_0(z)| \lesssim e^{R_{0,1}|z_1|^{\frac {2}{2\tau -1}}+\cdots +R_{0,d}|z_d|
^{\frac {2}{2\tau -1}} },
\end{equation}
where
\begin{equation}\label{Eq:MainResult3Eq3}
R_0 =
\frac {2\tau -1}{2}
\left (
\frac {2r_0}{2\tau +1}
\right )
^{{{\frac {2\tau +1}{2\tau -1}}}}
t_1^{-\frac {4}{2\tau -1}}
\text ;
\end{equation}

\vrum

\item Let $r_0,R_0\in \rr d_+$ be such that $r_0<r$ and
\begin{equation}\label{Eq:MainResult3Eq5}
R_0 =
\frac {2\tau -1}{2}
\left (
\frac {2r_0}{2\tau +1}
\right )
^{{{\frac {2\tau +1}{2\tau -1}}}}
t_2^{-\frac {4}{2\tau -1}},
\end{equation}
$F_0\in A(\cc d)$ be such that
\eqref{Eq:MainResult3Eq2} holds
and let $F=\Pi _A(F_0\cdot \chi )$. Then $F\in A(\cc d)$ and
satisfies \eqref{Eq:MainResult3Eq2}
for some $r_0\in \rr d_+$ such that $r_0<r$.
\end{enumerate}
\end{prop}

\par

Theorem \ref{Thm:MainResult6} follows from the following two propositions,
where the first one concerns the case when $s>0$ and the second
one make a more detailed explanation of the case $s=0$, i.{\,}e. the case
of analytic polynomials.

\par

\begin{prop}\label{Prop:MainResult6}
Let $s\in (0,\frac 12)$, $r\in \rr d_+$, and let $\chi \in \maclR ^\infty (\cc d)$.
Then the following is true:
\begin{enumerate}
\item Suppose $F\in A(\cc d)$ satisfies
\begin{equation}\label{Eq:MainResult6A}
|F(z)| \lesssim e^{r_{0,1} (\log \eabs {z_1})^{\frac 1{1-2s}} +\cdots +
r_{0,d} (\log \eabs {z_d})^{\frac 1{1-2s}}}
\end{equation}
for some $r_0\in \rr d_+$ such that $r_0<r$. Then there is an
$F_0\in A(\cc d)$ such that $F=\Pi _A(F_0\cdot \chi )$ and
\begin{equation}\label{Eq:MainResult6B}
|F_0(z)| \lesssim e^{r_{0,1} (\log \eabs {z_1})^{\frac 1{1-2s}} +\cdots +
r_{0,d} (\log \eabs {z_d})^{\frac 1{1-2s}}}
\end{equation}
for some $r_0\in \rr d_+$ such that $r_0<r$;

\vrum

\item Suppose $F_0\in A(\cc d)$ satisfies \eqref{Eq:MainResult6B}
for some $r_0\in \rr d_+$ such that $r_0<r$, and let
$F=\Pi _A(F_0\cdot \chi )$. Then $F\in A(\cc d)$ and satisfies
\eqref{Eq:MainResult6A} for some $r_0\in \rr d_+$ such that $r_0<r$.
\end{enumerate}
\end{prop}

\par

\begin{prop}\label{Prop:MainResult6No2}
Let $\chi \in \maclR ^\infty (\cc d)$ and let $N\ge 0$ be an integer.
Then the following is true:
\begin{enumerate}
\item Suppose $F\in A(\cc d)$ is given by
\begin{equation}\label{Eq:MainResult6ANo2}
F(z) = \sum _{|\alpha |\le N} c(F,\alpha ) z^\alpha ,
\end{equation}
where $\{c(F,\alpha ) \} _{|\alpha |\le N} \subseteq \mathbf C$.
Then there is an
$F_0\in A(\cc d)$ such that $F=\Pi _A(F_0\cdot \chi )$ and
\begin{equation}\label{Eq:MainResult6BNo2}
F_0(z) = \sum _{|\alpha |\le N} c(F_0,\alpha ) z^\alpha ,
\end{equation}
where $\{ c(F_0,\alpha ) \} _{|\alpha |\le N} \subseteq \mathbf C$
and satisfies $c(F_0,\alpha )=0$ when $c(F,\alpha )=0$;

\vrum

\item Suppose $F_0\in A(\cc d)$ satisfies \eqref{Eq:MainResult6BNo2}
for some $\{c(F_0,\alpha ) \} _{|\alpha |\le N} \subseteq \mathbf C$,
and let $F=\Pi _A(F_0\cdot \chi )$. Then $F\in A(\cc d)$ and satisfies
\eqref{Eq:MainResult6ANo2} for some $\{ c(F,\alpha )
\} _{|\alpha |\le N}
\subseteq \mathbf C$
such that $c(F,\alpha )=0$ when $c(F_0,\alpha )=0$.
\end{enumerate}
\end{prop}

\par

\subsection{Preparing results and their proofs}
For the proofs of Propositions \ref{Prop:MainResult1}--\ref{Prop:MainResult6}
and thereby of Theorems \ref{Thm:MainResult1}--\ref{Thm:MainResult6}
we need some preparatory results. Because the proof of
Proposition \ref{Prop:MainResult1} needs some room, we put parts
of the statement in the following separate proposition. At the same time we slightly
refine some parts concerning the image of compactly supported elements
in $L^\infty$ under the map $\Pi _A$.

\par

\begin{prop}\label{Prop:InclUltraDist}
Let $s>1$, $r_0,r\in \rr d_+$ be such that $r_0<r$ and suppose
%that one of the conditions are fulfilled:
%\begin{enumerate}
%\item $F_0\in \maclE _s'(D_{r_0}(0))$;
%
%\vrum
%
%\item $F_0\in \maclE _{0,s}'(D_{r_0}(0))$;
%
%\vrum
%
%\item $F_0\in \mascE '(D_{r_0}(0))$
%
%\vrum
%
%\item $F_0\in L^\infty (\overline D_{r}(0))$.
%\end{enumerate}
that either
$$
F_0\in \maclE _s'(D_{r_0}(0)),
\quad
F_0\in \maclE _{0,s}'(D_{r_0}(0)),
\quad
F_0\in \mascE '(D_{r_0}(0))
\quad \text{or}\quad
F_0\in L^\infty (\overline {D_{r}(0)}).
$$
Then $F=\Pi _AF_0 \in A(\cc d)$ and satisfies
$$
|F(z)| \lesssim e^{r_1|z_1|+\cdots +r_d|z_d|}.
$$
\end{prop}

\par

\begin{proof}
By the inclusions
$$
\mascE \hookrightarrow
\maclE _{0,s+\ep}' \hookrightarrow \maclE _s'
\hookrightarrow \maclE _{0,s}'
$$
when $\ep >0$, it suffices to consider
the case when $F_0\in \maclE _s'(D_{r_0}(0))$ or
$F_0\in L^\infty (\overline {D_{r}(0)})$ hold.

\par

Let $r_2=r$. First suppose that $F_0\in \maclE _s'(D_{r_0}(0))$ holds,
choose $r_1\in \rr d_+$ such that
$r_0<r_1<r_2$, $\Psi (y,\eta )=e^{-(|y|^2+|\eta |^2)}$ and let
$\Phi _z(y,\eta )=e^{(z,y+i\eta )}$.
By identifying $\cc d$ with $\rr {2d}$ and using the fact that
$F_0\in \maclE _s'(D_{r_0}(0))$ we obtain
\begin{equation}\label{Eq:CompRepKernel}
|\Pi _AF_0(z)| = \pi ^{-d}|\scal {F_0}{\Phi _z\Psi}|
\lesssim
\sup _{\alpha \in \nn {2d}}
\left (
\frac {\nm {D^\alpha (\Phi _z\Psi )}{L^\infty (D_{r_1}(0))}}
{h_1^{|\alpha |}\alpha !^s}
\right )
\end{equation}
for every $h_1>0$. We also have $\Psi \in \maclE _{1/2}(\rr {2d})
\hookrightarrow \maclE _{0,s}(\rr {2d})$, which implies
$$
\nm {D^\alpha \Psi}{L^\infty (D_{r_1}(0))}
\lesssim
h_2^{|\alpha |}\alpha !^s
$$
for every $h_2>0$. Furthermore,
$$
|D^\alpha \Phi _z(y,\eta )| = |m_\alpha (z) e^{(z,y+i\eta )}|
\le
|m_\alpha (z) |e^{r_{1,1}|z_1|+\cdots +r_{1,d}|z_d|},\quad
y+i\eta \in D_{r_1}(0),
$$
where
$$
m_\alpha (z)= \prod _{j=1}^dz_j^{\alpha _j+\alpha _{d+j}},
\qquad z\in \cc d,\ \alpha \in \nn {2d}.
$$

\par

By choosing $h_1=4$ and $h_2=1$ above, and letting $y+i\eta \in D_{r_1}(0)$,
Leibnitz rule gives
\begin{multline}\label{Eq:FactorialComp}
4^{-|\alpha |}\alpha !^{-s}
|D^\alpha (\Phi _z\Psi)(y,\eta )|
\\[1ex]
\le
4^{-|\alpha |}\alpha !^{-s}\sum _{\gamma \le \alpha}
{\alpha \choose \gamma}|D^\gamma \Phi _z(y,\eta )|
\, |D^{\alpha -\gamma}\Psi (y,\eta)|
\\[1ex]
\lesssim 4^{-|\alpha |}\alpha !^{-s}
\left (
\sum _{\gamma \le \alpha}
{\alpha \choose \gamma}
|m_\gamma (z) |(\alpha -\gamma )!^s
\right )
e^{r_{1,1}|z_1|+\cdots +r_{1,d}|z_d|}
\\[1ex]
\lesssim 4^{-|\alpha |}
\left (
\sum _{\gamma \le \alpha}
{\alpha \choose \gamma}
|m_\gamma (z) |\gamma !^{-s}
\right )
e^{r_{1,1}|z_1|+\cdots +r_{1,d}|z_d|}
\\[1ex]
\le
\sup _{\gamma \le \alpha}
\left (
\prod _{j=1}^d\frac {|z_j|^{\gamma _j}}{\gamma _j!^s}
\cdot
\frac {|z_j|^{\gamma _{d+j}}}{\gamma _{d+j}!^s}
\right )
e^{r_{1,1}|z_1|+\cdots +r_{1,d}|z_d|}.
\end{multline}
In the last inequality we have used that the number of terms in the
sums are bounded by $2^{|\alpha|}$, and that ${n \choose k}\le 2^k$
when $n,k$ are non-negative integers such that $k\le n$.

\par

By combining \eqref{Eq:FactorialComp} with the estimate
$$
\frac {|z_j|^{\gamma _j}}{\gamma _j!^s}
=
\left (
\frac {(|z_j|^{\frac 1s})^{\gamma _j}}{\gamma _j!}
\right )^s
\le
e^{s|z_j|^{\frac 1s}},
$$
we get
\begin{multline*}
\sup _{\alpha \in \nn {2d}}
\left (
\frac {\nm {D^\alpha (\Phi _z\Psi)}{L^\infty (D_{r_1}(0))}}
{h_1^{|\alpha |}\alpha !^s}
\right )
%\\[1ex]
\lesssim
e^{2s(|z_1|^{\frac 1s}+\cdots +|z_d|^{\frac 1s})}
e^{r_{1,1}|z_1|+\cdots +r_{1,d}|z_d|}
\\[1ex]
\lesssim
e^{r_{2,1}|z_1|+\cdots +r_{2,d}|z_d|}.
\end{multline*}
In the last inequality we have used the fact that $r_1<r_2$ and $s>1$.
From the latter estimate and \eqref{Eq:CompRepKernel} we obtain
$$
|\Pi _AF_0(z)|\lesssim e^{r_{2,1}|z_1|+\cdots +r_{2,d}|z_d|},
$$
and the result follows when $F_0\in \maclE _s'(D_{r_0}(0))$.

\par

Suppose instead that $F_0\in L^\infty (Q)$ holds, where
$Q=\overline {D_{r_2}(0)}\subseteq \cc d$, and let
$Q_j=\overline {D_{r_{2,j}}(0)}\subseteq \mathbf C$. Then
\begin{multline*}
|\Pi _AF_0(z)| \lesssim \int _{Q} |F_0(w)||e^{(z,w)}|\, d\lambda (w)
\\[1ex]
\le
\nm {F_0}{L^\infty}\prod _{j=1}^d
\left (
\int _{Q_j} e^{|z_j||w_j|}\, d\lambda (w_j)
\right )
\lesssim
e^{r_{2,1}|z_1|+\cdots +r_{2,d}|z_d|},
\end{multline*}
and the result follows in this case as well.
\end{proof}

\par

In the next lemma we give options
on compactly supported functions which are mapped on the basic monomials,
$e_\alpha$ by the operator $\Pi _A$.

\par

\begin{lemma}\label{Lemma:HermiteGroch}
Let $t_1,t_2\in \rr d_+$ be such that $t_1\le t_2$, $\chi \in
\mathcal R^\infty _{t_1,t_2}(\cc d)$, and let $\chi _0$ be such
that $\chi _0(|z_1|,\dots ,|z_d|)=\chi (z_1,\dots ,z_d)$. If
$$
F_{\alpha ,\chi}(z) = \varsigma _\alpha z^\alpha \chi (z),
$$
with
\begin{equation}\label{Eq:varsigmaDef}
\varsigma _\alpha = 2^{-d}\alpha !^{\frac 12}
\left (
\int _{\Delta _{t_2}}
\chi _0(u)e^{-|u|^2}u^{2\alpha}u_1\cdots u_d\, du
\right )^{-1},\quad \Delta _t = \sets {u\in \rr d_+}{u\le t},
\end{equation}
then the following is true:
\begin{enumerate}
\item $\Pi _AF_{\alpha ,\chi }=e_\alpha$;

\vrum

\item for some constant $C>0$ which only depends on $\nm \chi{L^\infty}$,
$c$ in Definition \ref{Def:RClass} and the dimension $d$, it holds
\begin{equation}\label{Eq:Estvarsigma}
C^{-1}
\left (
\prod _{j=1}^d
t_{2,j}^{-2}(\alpha _j+1)
\right ) t_2^{-2\alpha} \alpha !^{\frac 12}
\le
\varsigma _\alpha
\le
C
e^{|t_1|^2}
\left (
\prod _{j=1}^d
t_{1,j}^{-2}(\alpha _j+1)
\right ) t_1^{-2\alpha}\alpha !^{\frac 12}
\end{equation}
\end{enumerate}
\end{lemma}

\par

\begin{proof}
By using polar coordinates in each complex variable when integrating we get
\begin{multline}\label{Eq:PiFalpha}
(\Pi _AF_{\alpha ,\chi})(z) = \pi ^{-d}\varsigma _\alpha
\int _{\cc d}w^\alpha \chi (w) e^{(z,w)-|w|^2}\, d\lambda (w)
\\[1ex]
=
\pi ^{-d}\varsigma _\alpha
\int _{\Delta _{t_2}} I _\alpha (u,z)e^{-|u|^2}u^\alpha
\chi _0(u)u_1\cdots u_d\, du,
\end{multline}
where
\begin{align}
I _\alpha (u,z)
&=
\int _{[0,2\pi )^d}e^{i\scal \alpha \theta}
\left (
\prod _{j=1}^d
e^{z_ju_je^{-i\theta _j}}
\right )
\, d\theta
=
\prod _{j=1}^d I_{\alpha _j}(u_j,z_j)
\label{Eq:IalphaDef}
\intertext{with}
I_{\alpha _j}(u_j,z_j)
&=
\int _0^{2\pi}e^{i\alpha _j\theta _j}e^{z_ju_je^{-i\theta _j}}\, d\theta _j.
\notag
\end{align}
By Taylor expansions we get
\begin{multline*}
I_{\alpha _j}(u_j,z_j)
=
\int _0^{2\pi}e^{i\alpha _j\theta _j}
\left (
\sum _{k=0}^\infty
\frac {z_j^ku_j^ke^{-ik\theta _j}}{k!}
\right )
\, d\theta _j
\\[1ex]
=
\sum _{k=0}^\infty
\left (
\left (
\int _0^{2\pi}e^{i(\alpha _j-k)\theta _j}
\, d\theta _j 
\right )
\frac {z_j^ku_j^k}{k!}
\right )
=
\frac {2\pi z_j^{\alpha _j}u_j^{\alpha _j}}{\alpha _j!}
\end{multline*}

\par

By inserting this into \eqref{Eq:PiFalpha} and \eqref{Eq:IalphaDef} we get
\begin{multline*}
(\Pi _AF_{\alpha ,\chi})(z) = \pi ^{-d}\varsigma _\alpha
\int _{\Delta _{t_2}} 
(2\pi )^d\frac {u^\alpha z^\alpha}{\alpha !}
e^{-|u|^2}u^\alpha \chi _0(u)u_1\cdots u_d\, du
\\[1ex]
=
\left (
2^d \varsigma _\alpha 
\alpha !^{-\frac 12}
\int _{\Delta _{t_2}} 
e^{-|u|^2}u^{2\alpha} \chi _0(u)u_1\cdots u_d\, du
\right )
e_\alpha (z)
= e_\alpha (z)
\end{multline*}
and (1) follows.

\par

Since $\chi$ is non-negative and fullfils $\chi \ge c$ on $D_{t_1}(0)$ we get
\begin{multline*}
\varsigma _\alpha
\lesssim
e^{|t_1|^2}\alpha !^{\frac 12}
\left (
\int _{\Delta _{t_1}}u^{2\alpha}u_1\cdots u_2\, du
\right )^{-1}
\\[1ex]
=
e^{|t_1|^2}\alpha !^{\frac 12}
\prod _{j=1}^d
\left (
\frac {t_{1,j}^{2\alpha _j+2}}{2\alpha _j+2}
\right )^{-1}
\asymp
e^{|t_1|^2}\alpha !^{\frac 12}
\left (
\prod _{j=1}^d (t_{1,j}^{-2}(\alpha _j+1)),
\right )
t_1^{-2\alpha}
\end{multline*}
which gives the right inequality in \eqref{Eq:Estvarsigma}. By the
support properties of $\chi$ we also have
$$
\varsigma _\alpha
\gtrsim
\alpha !^{\frac 12}
\left (
\int _{\Delta _{t_2}}u^{2\alpha}u_1\cdots u_2\, du
\right )^{-1},
$$
and the left inequality in \eqref{Eq:Estvarsigma} follows by similar arguments, and
(2) follows.
\end{proof}

\par

The next lemma shows that we may estimate entire functions by
different Lebesgue norms. We omit the proof, since the result
follows from \cite[Theorem 3.2]{To11}.

\par

\begin{lemma}\label{Lemma:EstimateL2LInfty}
Suppose $s,\tau \in \mathbf R$ and $r,r_0\in \rr d_+$ are
such that
$$
s< \frac 12,\quad \tau >-\frac 12
\quad \text{and}\quad
r_0<r.
$$
Let $p,q\in [1,\infty]$, $F\in A(\cc d)$ and set
\begin{align*}
M_{1,r}(z)
&=
r_1|z_1|^{\frac {2}{2\tau +1}}+\cdots
+r_d|z_d|^{\frac {2}{2\tau +1}}
\intertext{and}
M_{2,r}(z)
&=
r_1(\log \eabs {z_1})^{\frac 1{1-2s}}
+\cdots +
r_d(\log \eabs {z_d})^{\frac 1{1-2s}}.
\end{align*}
Then
\begin{align*}
    \nm{F \cdot e^{-M_{j,r} }}{L^p(\cc{d} )} 
    &\lesssim 
    \nm {F \cdot e^{-M_{j,r_0} }}{L^q(\cc{d} )} .
%\intertext{and} 
%    \nm {F \cdot e^{-M_{j,r} }}{L^{\infty}(\cc{d} )} 
%    &\lesssim
%    \nm {F \cdot e^{-M_{j,r_0} }}{L^{2}(\cc{d} )},\qquad j=1,2.
    \end{align*}
\end{lemma}

\par

The next lemma relates Lebesgue estimates of entire functions with
estimates on corresponding Taylor coefficients. 
Here we let the Gamma function on $\rr d_+$ be defined by
$$
\Gamma _{\! d}(x_1,\dots ,x_d) = \prod _{j=1}^d\Gamma (x_j),
$$
where $\Gamma$ is the Gamma function on $\mathbf R_+$.

\par

\begin{lemma}\label{Lemma:EstFtimesEfunction}
Let $\tau > -\frac 12$, $M_{1,r}$ be the same as in Lemma \ref{Lemma:EstimateL2LInfty},
$\omega (z)= e^{\frac 12|z|^2-M_{1,r}(z)}$, $j=1,2$,
$z\in \cc d$, and let $\alpha _0 = (1,\dots ,1)\in \nn d$. 
Also let
\begin{align*}
\vartheta _{r}(\alpha ) 
&=
\left(
(2\tau +1)
(2r)^{-(2\tau +1)(\alpha +\alpha _0)}
\left (
\frac {\Gamma
\left (
(2\tau +1)(\alpha +\alpha _0)
\right )}{\alpha !}
\right )
\right )
^{\frac 12} .
\end{align*}
If $F \in A (\cc d)$ is given by \eqref{Eq:PowerSeriesExp},
then
$$
\nm{F \cdot e^{-M_{1,r}}}{L^2(\cc{d})} 
=
\pi ^{\frac d2}\left (
\sum_{\alpha \in \nn d}
|c(F,\alpha )\vartheta _{r}(\alpha )|^2
\right )^{\frac 12},
$$
and $\maclA _{[\vartheta _{r}]}^2(\cc d) = A^2_{(\omega )}(\cc d)$
with equality in norms.
\end{lemma}

\par

\begin{proof}
Since
$$
e^{-M_{1, r}(z)}= \prod 
_{j=1}^de^{-r_j|z_j|^{\frac{2}{2\tau+1}}},
$$
Lemma \ref{Lemma:SplitWeights} shows that we may assume that
$d=1$, giving that $r=r_1$ and $\alpha _0=1$.

\par

In view of Theorem \ref{cA2sCharExt} we have
$\maclA _{[\vartheta ]}^2(\cc d)= A_{(\omega _1)}^2(\cc d)$
with equality in norms, when
\begin{align*}
    	\vartheta(\alpha) = 
    	\left( 
    	\frac{1}{\alpha!}
    	\int_0^{\infty} e^{-2r t^{\frac{1}{2\tau +1}} } 
    	t^{\alpha} \, dt
    	\right)
    	^{\frac 12},
\end{align*}
provided it is verified that \eqref{varthetaCond} holds for every $r>0$.
By taking $u= 2rt^{\frac{1}{2\tau +1}}$ as new variables of integration we obtain
\begin{align*}
    	\vartheta(\alpha)
    	&=
    	\left(
    	(2\tau +1)
    	(2r)^{-(2\tau +1)(\alpha +1)}
        \frac 1{\alpha !}\int_0^{\infty} e^{-u} 
        u^{\alpha(2\tau +1)+2\tau}
        \, du
        \right )^{\frac 12}
\\[1ex]
        &=
        \left(
        (2\tau +1)
        (2r)^{-(2\tau +1)(\alpha +1)}
        \left (
        \frac {\Gamma
        \left (
        (2\tau +1)(\alpha +1)
        \right )}{\alpha !}
        \right )
        \right )
        ^{\frac 12} .
\end{align*}
Since $\Gamma ( (2\tau +1)(\alpha +1))\gtrsim r^{|\alpha |}$
for every $r>0$, by Stirling's formula, \eqref{varthetaCond}
holds for every $r>0$. This implies that $\vartheta =\vartheta _{1,r}$,
and the result follows.
\end{proof}

\par

We also need the following version of Stirling's formula.

\par

\begin{lemma}\label{Lemma:GammaReformulation1}
Let $\alpha \ge 0$ be an integer and let $\tau >-\frac 12$. Then
\begin{align}
\frac{\displaystyle{\Gamma 
\left ( 
(2\tau +1)(\alpha+1)
\right )}}
{\alpha !}
&\asymp
\left (
2\tau +1
\right )
^{(2\tau +1) \cdot \alpha}
(\alpha +1)^{\tau}\cdot \alpha !^{2\tau}.
\label{Eq:Stirling1}
\end{align}
\end{lemma}

\par

Lemma \ref{Lemma:GammaReformulation1} follows by
repeated applications of Stirling's formula and the
standard limit
$$
\lim _{t\to \infty}
\left (
1+\frac xt
\right )
^{t} = e^x
$$
for every $x\in \mathbf R$. In order to be self-contained
we present the arguments.

\par

\begin{proof}
The result is obviously true for $\alpha =0$. For $\alpha \ge 1$
we have $\alpha ^{\tau} \asymp (\alpha +1)^{\tau}$.
A combination of the latter relations and Stirling's formula gives
\begin{multline}\label{Eq:GammaReformulation1}
\frac{\displaystyle{\Gamma 
\left ( 
(2\tau +1)(\alpha+1) 
\right )}}
{\alpha !}
\asymp
 \frac{\displaystyle{ \left ( 
(2\tau +1) \alpha + 2\tau 
 \right )
 ^{(2\tau +1) \cdot \alpha +2\tau +\frac 12}}}
 {\displaystyle{e^{\alpha + 2\tau (\alpha +1)}}}
 \cdot
 \frac{e^\alpha}{\alpha^{\alpha + \frac 12}}
\\[1ex]
=
\frac{\displaystyle{
\left (
2\tau +1
\right )
^{(2\tau +1) \cdot \alpha + 2\tau
+\frac 12}
\left (
1 + \frac{2\tau}{(2\tau +1)\alpha}
\right )
^{(2\tau +1) \cdot \alpha + 2\tau
+\frac 12}
\alpha ^{2\tau (\alpha +1)}
}}
{e^{2\tau (\alpha + 1)}}
\\[1ex]
\asymp
\left (
2\tau +1
\right )
^{(2\tau +1) \cdot \alpha}
\frac{\displaystyle{
\alpha ^{2\tau (\alpha +1)}
}}
{e^{2\tau \alpha}}
\cdot
\left (
1+\frac {2\tau}{(2\tau +1)\alpha}
\right )^{(2\tau +1)\alpha}
\\[1ex]
\asymp
\left (
2\tau +1
\right )
^{(2\tau +1) \cdot \alpha}
\frac{\displaystyle{
\alpha ^{2\tau (\alpha +1)}
}}
{e^{2\tau \alpha}}
\asymp
\left (
2\tau +1
\right )
^{(2\tau +1)\cdot \alpha}
(\alpha +1)^{\tau}\alpha !^{2\tau}.
\end{multline}
In the forth relation we have used the fact that
$$
\left (
1+ \frac {x_1}\alpha
\right )^ {x_2\alpha},\qquad x_1,x_2\in \mathbf R_+,
$$
increases with $\alpha$ and has the limit $e^{x_1x_2}$ as $\alpha$
tends to infinity. This gives the result.
\end{proof}

\par

Proposition \ref{Prop:MainResult1} essentially follows from the following lemma.

\par

\begin{lemma}\label{Lemma:Omega34Equiv}
Let $\tau >-\frac 12$, $r_0,r\in \rr d_+$ be such that $r_0<r$
and $F\in A(\cc d)$. Then the following is true:
\begin{enumerate}
\item If
\begin{equation}\label{Eq:CoeffEst1}
|c(F,\alpha )|
\lesssim
\left (
\frac {2 r_0}{2\tau +1}
\right )
^{\frac {2\tau +1}{2} \cdot \alpha}\alpha !^{-\tau},
\end{equation}
then
\begin{equation}\label{Eq:FuncEst1}
|F(z)|\lesssim e^{r_1|z_1|^{\frac {2}{2\tau +1}} +\cdots + r_d|z_d|
^{\frac {2}{2\tau +1}}} \text ;
\end{equation}

\vrum

\item If
\begin{equation}\label{Eq:FuncEst2}
|F(z)|\lesssim e^{r_{0,1}|z_1|^{\frac {2}{2\tau +1}} +\cdots + r_{0,d}|z_d|
^{\frac {2}{2\tau +1}}}
\end{equation}
then
\begin{equation}\label{Eq:CoeffEst2}
|c(F,\alpha )|
\lesssim
\left (
\frac {2 r}{2\tau +1}
\right ) ^{\frac {2\tau +1}{2} \cdot \alpha}\alpha !^{-\tau}
\end{equation}
\end{enumerate}
\end{lemma}

\par

\begin{proof}
Let $\vartheta _{r}$ be the same as in Lemma
\ref{Lemma:EstFtimesEfunction}.
First we prove (1). Suppose that \eqref{Eq:CoeffEst1} holds, let
$r_1\in \rr d_+$ be such that
$r_0 <r_1<r$
and let $\alpha _0=(1,\dots ,1)\in \nn d$. Also let $M_{1,r}$  be the same as in
Lemma \ref{Lemma:EstimateL2LInfty}. Then Lemmata
\ref{Lemma:EstimateL2LInfty} and \ref{Lemma:EstFtimesEfunction} give
\begin{multline*}
\nm {F\cdot e^{-M_{1,r}}}{L^\infty (\cc d)}
\lesssim
\nm {F\cdot e^{-M_{1,r_1}}}{L^2 (\cc d)}
\\[1ex]
\asymp
\left (
\sum_{\alpha \in \nn d}
        \left |
        c(F,\alpha ) \vartheta _{r_1}(\alpha )
        \right |^2
        \right )
        ^{\frac 12}
\\[1ex]
=
\left (
\sum_{\alpha \in \nn d}
        \left |
        c(F,\alpha ) \vartheta _{r_0}(\alpha )
        \left (
        \frac {r_0}{r_1}
        \right )^{(\alpha +\alpha _0)\frac {2\tau +1}{2}}
        \right |^2
        \right )
        ^{\frac 12}
\\[1ex]
\lesssim
\sup _{\alpha \in \nn d}\left |
c(F,\alpha ) \vartheta _{r_0}(\alpha )
\prod _{j=1}^d(\alpha _j+1)^{-\frac \tau{2}}
\right |
\\[1ex]
\asymp
\sup _{\alpha \in \nn d}
\left (
|c(F,\alpha )|
\left (
\frac {2\tau +1}{2r_0}
\right )^{\frac {2\tau +1}{2}\cdot \alpha} \alpha !^{\tau}
\right ).
\end{multline*}
Here the second inequality follows from the fact that
$$
\sum _{\alpha \in \nn d} 
\left (
\left (
\prod _{j=1}^d
(\alpha _j+1)^{\frac \tau{2}}
\right )
\left (
\frac {r_0}{r_1} 
\right )
^{(2\tau +1) \cdot (\alpha +\alpha _0)}
\right )
$$
is convergent since $r_0<r_1$, and the fifth relation
follows from Lemma \ref{Lemma:GammaReformulation1}.
This implies that \eqref{Eq:FuncEst1} holds and (1) follows.

\par

Next assume that \eqref{Eq:FuncEst2} holds. By
Lemmata \ref{Lemma:EstimateL2LInfty} and
\ref{Lemma:EstFtimesEfunction} we get
\begin{multline*}
\nm {F\cdot e^{-M_{1,r_0}}}{L^\infty (\cc d)}
\gtrsim
\nm {F\cdot e^{-M_{1,r}}}{L^2 (\cc d)}
\\[1ex]
=
\left (
\sum_{\alpha \in \nn d}
        \left |
        c(F,\alpha ) \vartheta _{r}(\alpha )
        \right |^2
        \right )
        ^{\frac 12}
%\\[1ex]
\ge
\sup _{\alpha \in \nn d}
\left (
\left |
c(F,\alpha )
\right | \vartheta _{r}(\alpha )
\right )
\\[1ex]
\gtrsim
\sup _{\alpha \in \nn d} 
\left (
|c(F,\alpha ) | 
\left( 
\frac{2\tau +1}{2r}
\right )
^{\frac{2\tau +1}{2}\alpha}
\prod _{j=1}^d(\alpha _j+1)^{\frac{\tau}{2}}
\alpha !^{\tau} 
\right )
\\[1ex]
\ge
\sup _{\alpha \in \nn d} 
\left (|c(F,\alpha ) | 
\left( 
\frac{2\tau +1}{2r}
\right )
^{\frac{2\tau +1}{2}\alpha}
\alpha !^{\tau}
\right ),
\end{multline*}
where the third inequality follows from
Lemma \ref{Lemma:GammaReformulation1}. This gives (2).
\end{proof}

\par

Proposition \ref{Prop:MainResult6} mainly follows from the following
result.

\par

\begin{lemma}\label{Lemma:MainResult6A}
Let $r,r_0\in \rr d_+$ be such that $r_0<r$, $s\in (0,\frac 12)$ and let
$F\in A(\cc d)$. Then the following is true:
\begin{enumerate}
\item if \eqref{Eq:MainResult6A} holds, then
\begin{equation}\label{Eq:LemmaMainResult6B}
|c(F,\alpha )|\lesssim e^{-(R_{1}|\alpha _1|^{\frac 1{2s}}
+\cdots + R_{d}|\alpha _d|^{\frac 1{2s}})},\quad
R=s \left (
\frac {1-2s}r
\right )
^{\frac {1-2s}{2s}}\text ;
\end{equation}

\vrum

\item if $R_0\in \rr d_+$ is given by
\begin{equation}\label{Eq:LemmaMainResult6C}
|c(F,\alpha )|\lesssim e^{-(R_{0,1}|\alpha _1|^{\frac 1{2s}}
+\cdots + R_{0,d}|\alpha _d|^{\frac 1{2s}})},\quad
R_0=s \left (
\frac {1-2s}{r_0}
\right )
^{\frac {1-2s}{2s}},
\end{equation}
then \eqref{Eq:MainResult6A} holds with $r$ in place of $r_0$.
\end{enumerate}
\end{lemma}

\par

\begin{proof}
Let $r_1,r_2\in \rr d_+$ be such that $r_1<r<r_2$.
By Lemma \ref{Lemma:SplitWeights} we may assume that
$d=1$, and by Lemma \ref{Lemma:EstimateL2LInfty} the result
follows if we prove
\begin{multline}\label{Eq:SmallParSqueze}
\sup _{\alpha \in \mathbf N}
\left (
|c(F,\alpha ) e^{R_{2}\cdot \alpha ^{\frac 1{2s}}}|^2
\right )
\lesssim
\int _{\mathbf C}|F(z) e^{-r(\log \eabs z)^{\frac 1{1-2s}}}|^2\, d\lambda (z)
\\[1ex]
\lesssim
\sup _{\alpha \in \mathbf N}
\left (
|c(F,\alpha ) e^{R_{1}\cdot \alpha ^{\frac 1{2s}}}|^2
\right ),
\end{multline}
where $R_1,R_2\in \mathbf R_+$ satisfy
$$
R_j =
\left (
\frac {1-2s}{r_j}
\right )
^{\frac {1-2s}{2s}},\quad j=1,2.
$$

\par

By Theorem \ref{cA2sCharExt} we have
$$
\int _{\mathbf C}|F(z) e^{-r(\log \eabs z)^{\frac 1{1-2s}}}|^2\, d\lambda (z)
\asymp
\sum _{\alpha \in \mathbf N}
|c(F,\alpha )\vartheta _{r}(\alpha )|^2,
$$
where
$$
\vartheta _r(\alpha )=
\left (
\frac \pi{2\alpha !}
\int _0^\infty e^{-r(\log \eabs t)^{\frac 1{1-2s}}}t^\alpha \, dt
\right )^{\frac 12}.
$$

\par

Let
$$
\theta = \frac 1{1-2s}
\quad \text{and}\quad
g_{r,\alpha} (t) = e^{-r(\log t)^\theta}t^\alpha .
$$
In order to prove \eqref{Eq:SmallParSqueze},
we need to prove
\begin{equation}\label{Eq:varthetaIneq}
e^{R_2\alpha ^{\frac 1{2s}}}
\lesssim
\vartheta _{r}(\alpha )
\lesssim
e^{R_1\alpha ^{\frac 1{2s}}},
\end{equation}
which shall be reached by modifying the proof of (15)
in \cite{FeGaTo2}.

\par

We have
\begin{multline*}
\vartheta _{r}(\alpha )^2
\lesssim
\int _e^\infty e^{(r_1-r)(\log t)^\theta}g_{r_1,\alpha} (t)\, dt
\\[1ex]
\lesssim
\sup _{t\ge e}(g_{r_1,\alpha} (t))
\int _e^\infty e^{(r_1-r)(\log t)^\theta}\, dt
\asymp
\sup _{t\ge e}(g_{r_1,\alpha} (t)).
\end{multline*}

\par

By straight-forward computations it follows that
$g_{r,\alpha} (t)$ attains its global maximum for
\begin{equation}\label{Eq:tralphaFormula}
t_{r,\alpha} = \exp
\left (
\left (
\frac {\alpha }{\theta r}
\right )^{\frac {1-2s}{2s}}
\right ),
\end{equation}
and that
$$
g_{r_1,\alpha }(t_{r_1,\alpha} )
=
\exp \left (
{r_1^{1-\frac 1{2s}}}{\theta ^{-\frac 1{2s}}
(\theta -1)\alpha ^{\frac 1{2s}}}
\right )
=
e^{2R_1\cdot \alpha ^{\frac 1{2s}}},
$$
and the second inequality in \eqref{Eq:varthetaIneq} follows.

\par

In order to prove the first inequality in \eqref{Eq:varthetaIneq},
we claim that for some $c$ which is independent of $\alpha$ we have
\begin{equation}\label{Eq:t2alpaEst}
\left (
1-\frac c{t_{r_2,\alpha}}
\right )^\alpha
\gtrsim
e^{(r-r_2)(\log t_{r_2,\alpha})^\theta}.
\end{equation}
In fact, by \eqref{Eq:tralphaFormula} and Taylor's formula, 
\eqref{Eq:t2alpaEst} follows if we prove
\begin{equation}\label{Eq:t2alpaEst2}
1-\frac {\alpha c}{t_{r_2,\alpha}}
\gtrsim
e^{(r-r_2)((1-2s)\alpha /r_2)^{\frac 1{2s}}}
\end{equation}
for some $c>0$ which is independent of $\alpha$. In view of
\eqref{Eq:tralphaFormula} it follows that $t_{r_2,\alpha}$
tends to infinity as $\alpha$ turns to infinity, which implies
that the left-hand side of \eqref{Eq:t2alpaEst2} has the limit
one and the right-hand side has the limit zero as $\alpha$ tends to
infinity. Hence, \eqref{Eq:t2alpaEst} holds.

\par

By \eqref{Eq:t2alpaEst} and the fact that $\frac 1{2s}>1$ we get
\begin{multline*}
\vartheta _r(\alpha )^2\gtrsim \frac 1{\alpha !}
\int _{t_{r_2,\alpha}-c}^{t_{r_2,\alpha}}e^{-r(\log t)^\theta}t^\alpha
\, dt
\gtrsim
\frac 1{\alpha !}e^{-r(\log t_{r_2,\alpha})^\theta}
(t_{r_2,\alpha}-c)^\alpha
\\[1ex]
\gtrsim
\frac 1{\alpha !}e^{-r_2(\log t_{r_2,\alpha})^\theta}
t_{r_2,\alpha}^\alpha
=
\frac 1{\alpha !}e^{2R_2\cdot \alpha ^{\frac 1{2s}}}
\gtrsim
e^{2R_3\cdot \alpha ^{\frac 1{2s}}}.
\end{multline*}
%%
% let $c>0$ be chosen such that
% $$
% e^{-r(\log t)^\theta} \ge e^{-r_2(\log t_{2,\alpha})^\theta},
% \quad
% t_{2,\alpha}-c\le t\le t_{2,\alpha}.
% $$
% This is possible since
% $$
% e^{-r(\log t)^\theta} > e^{-r_2(\log t)^\theta}
% $$
% and that all these expressions depend continuously on $t$. By
% similar arguments as in the first part of the proof we get
% %%
% \begin{multline*}
% \vartheta _r(\alpha )^2\gtrsim \frac 1{\alpha !}
% \int _{t_{2,\alpha}-c}^{t_{2,\alpha}}e^{-r(\log t)^\theta}t^\alpha
% \, dt
% \gtrsim
% \frac 1{\alpha !}e^{-r_2(\log t_{2,\alpha})^\theta}
% t_{2,\alpha}^\alpha
% \\[1ex]
% =
% \frac 1{\alpha !}\exp \left (
% {r_2^{1-\frac 1{2s}}}{\theta ^{-\frac 1{2s}} (\theta -1)\alpha ^{\frac 1{2s}}}
% \right )
% =
% \frac 1{\alpha !}e^{2R_2\cdot \alpha ^{\frac 1{2s}}}
% \gtrsim
% e^{2R_3\cdot \alpha ^{\frac 1{2s}}}.
% \end{multline*}
% %%
% Here the last inequality follows from the fact that
% $\frac 1{2s}>1$.
This gives the result.
\end{proof}

\par

\subsection{Proofs of main results}

\par

Next we prove Proposition \ref{Prop:MainResult1} and thereby Theorem
\ref{Thm:MainResult1}.

\par

\begin{proof}[Proof of Proposition \ref{Prop:MainResult1}]
It is clear that (2) $\Rightarrow$ (3) $\Rightarrow$ (4) $\Rightarrow$ (5).
By Proposition
\ref{Prop:InclUltraDist} it follows that (5) implies (1). We need to prove
that (1) implies (2).

\par

Suppose (1) holds and let $r_4=r$. Choose $r_1,r_2,r_3\in \rr d_+$
such that $r_0<r_1< r_2\le r_3<r_4$ and $r_1r_4<r_2^2$,  $\chi \in
\maclR _{r_2,r_3}^\infty (\cc d)$, and let $F_{\alpha ,\chi}$ be as in Lemma
\ref{Lemma:HermiteGroch} with $r_2$ and $r_3$ in place of $r_1$ and $r_2$.
If $\tau = 1$, then
Lemma \ref{Lemma:Omega34Equiv} (2) gives
\begin{equation}\label{Eq:CoeffCondFlat1}
|c(F,\alpha )| \lesssim r_1^\alpha \alpha !^{-\frac 12}.
\end{equation}
Let
\begin{equation}\label{Eq:F0Def}
F_0(z)
=
\sum _{\alpha \in \nn d}c(F,\alpha ) \varsigma _\alpha z^\alpha .
\end{equation}
We claim that the series in \eqref{Eq:F0Def} is uniformly convergent
with respect to $z$ in $\overline{D_R(0)}$.

\par

In fact, if $|z_j|\le r_{4,j}$, then \eqref{Eq:Estvarsigma} gives
$$
|c(F,\alpha )\varsigma _\alpha z^\alpha |\lesssim \eabs \alpha ^d r_2^{-2\alpha}
r_1^\alpha r_4^\alpha = \eabs \alpha ^d \rho ^\alpha ,
$$
where $\rho \in \rr d_+$ satisfies $\rho _j= \frac {r_{1,j}r_{4,j}}{r_{2,j}^2}<1$.
Since $\sum _{\alpha \in \nn d}\eabs \alpha ^d\rho ^\alpha$ is convergent,
Weierstrass theorem
shows that \eqref{Eq:F0Def} is uniformly convergent and defines an analytic
function in $\overline {D_{r_{4,j}}(0)}$. Hence, $F_0\in A(\overline {D_{r_{4,j}}(0)})$.
Furthermore, by the support of $\chi$ we get
\begin{multline}\label{Eq:PiAF0chiComp}
\Pi _A(F_0\cdot \chi )
=
\Pi _A
\left (
\sum _{\alpha \in \nn d} c(F,\alpha )\varsigma _\alpha z^\alpha \cdot \chi
\right )
\\[1ex]
=
\sum _{\alpha \in \nn d}
c(F,\alpha )\varsigma _\alpha \Pi _A (z^\alpha \cdot \chi )
= \sum _{\alpha \in \nn d}c(F,\alpha )e_\alpha  = F.
\end{multline}
Hence (2) holds, and the proof is complete.
\end{proof}

\par

For future references we observe that if $\varsigma _\alpha$ and $\chi$
are the same as in Lemma \ref{Lemma:HermiteGroch}, then \eqref{Eq:F0Def}
shows that the relationship between $c(F,\alpha )$ and $c(F_0,\alpha )$ is given
by
\begin{equation}\label{Eq:CoeffF0FComp}
c(F_0,\alpha ) = c(F,\alpha )\varsigma _\alpha \alpha !^{\frac 12}.
\end{equation}

\par

\begin{proof}[Proof of Theorem \ref{Thm:MainResult2}]
The equivalence between (1) and (2) is clear. It is also obvious that (3)
implies (2) in view of Proposition \ref{Prop:MainResult1}, and that
(3) implies (4). We shall prove that (1) implies (3) and (4) implies (1).

\par

Suppose (1) holds. Then \eqref{Eq:CoeffCondFlat1} holds for every
$r_1\in \rr d_+$. Let $r,R\in \rr d_+$, $\chi \in \maclR _{r,R}^\infty (\cc d)$,
$F_{\alpha ,\chi}$ be as in Lemma \ref{Lemma:HermiteGroch} and let
$F_0$ be given by \eqref{Eq:F0Def}. By \eqref{Eq:CoeffCondFlat1} we have
$$
|c(F,\alpha ) \varsigma _\alpha | \lesssim \eabs \alpha ^dr^{-2\alpha}r_1^{\alpha}
$$
for every $r_1\in \rr d_+$, giving that
$$
|c(F,\alpha ) \varsigma _\alpha | \lesssim r_0^\alpha 
$$
for every $r_0\in \rr d_+$. This implies that the series in \eqref{Eq:F0Def}
is locally uniformly convergent with respect to $z$ and defines an entire function
on $\cc d$. Hence $F_0\in A(\cc d)$. Moreover, by \eqref{Eq:PiAF0chiComp}
it follows that $\Pi _A(F_0\chi )=F$, and we have proved that (1) implies (3).

\par

Next suppose that (4) holds. Then
$$
|c(F_0,\alpha )|\lesssim r^\alpha \alpha !^{\frac 12}
$$
for every $r\in \rr d_+$. By \eqref{Eq:Estvarsigma} and
\eqref{Eq:CoeffF0FComp} we get
\begin{equation*}
|c(F,\alpha )|= |c(F_0,\alpha )| \varsigma _\alpha ^{-1}\alpha !^{-\frac 12}
%\\[1ex]
\lesssim
r^\alpha \varsigma _\alpha ^{-1}
\lesssim
r^{\alpha}R^{2\alpha}\alpha !^{-\frac 12}.
\end{equation*}
Since $r\in \rr d_+$ can be chosen arbitrarily small we get
$$
|c(F,\alpha )|\lesssim r_0^\alpha \alpha !^{-\frac 12}
$$
for every $r_0\in \rr d_+$. This implies that $F\in \maclA _{0,\flat _1}(\cc d)$. That is,
we have proved that (4) implies (1), and the result follows.
\end{proof}

\par

Next we prove Proposition \ref{Prop:MainResult345}.
%, and thereby Theorem
%\ref{Thm:MainResult3}.

\par

\begin{proof}[Proof of Proposition \ref{Prop:MainResult345}]
Suppose \eqref{Eq:MainResult3Eq1} holds, and let $r_1,r_2\in \rr d_+$
be such that $r_0<r_1<r_2<r$.
Then Lemma \ref{Lemma:Omega34Equiv} gives
$$
|c(F,\alpha )|
\lesssim
\left (
\frac {2r_1}{2\tau +1}
\right )^{\frac {2\tau +1}{2}\cdot \alpha }
\alpha !^{-\tau}.
$$
By Lemma \ref{Lemma:HermiteGroch} and \eqref{Eq:CoeffF0FComp}
we get
\begin{multline*}
|c(F_0,\alpha )|
\lesssim 
\left (
\frac {2r_1}{2\tau +1}
\right )^{\frac {2\tau +1}{2}\cdot \alpha } \alpha !^{-\tau}
\eabs \alpha ^d\alpha ! t_1^{-2\alpha}
\\[1ex]
\lesssim
\left (
\frac {2r_2}{2\tau +1}
\right )^{\frac {2\tau +1}{2}\cdot \alpha } \alpha !^{-(\tau -1)}
t_1^{-2\alpha}
%\\[1ex]
=
\left (
\frac {2R_0}{2\tau -1}
\right )^{\frac {2\tau -1}{2}\cdot \alpha } \alpha !^{-(\tau -1)},
\end{multline*}
when
$$
R_0 = \frac {2\tau -1}{2}
\left (
\frac {2R_2}{2\tau +1}
\right )
^{\frac {2\tau +1}{2\tau -1}}t_1^{-\frac 4{2\tau -1}}.
$$
This proves (1).

\par

Suppose instead that $r_0,R_0,r_1,r_2\in \rr d_+$ are such that
\eqref{Eq:MainResult3Eq5} hold and $r_0<r_1<r_2<r$, and
that $F_0\in A(\cc d)$ satisfies \eqref{Eq:MainResult3Eq2}.
Also let $F=\Pi _A(F_0\cdot \chi )$.
Then
$$
\left (
\frac {2R_0}{2\tau -1}
\right )^{\frac {2\tau -1}{2}}
<
\left (
\frac {2r_1}{2\tau +1}
\right )^{\frac {2\tau +1}{2}}t_2^{-2}.
$$
By combining the latter estimate with Lemma \ref{Lemma:Omega34Equiv}
we get
\begin{equation*}
|c(F_0,\alpha )|
\lesssim 
\left (
\frac {2r_1}{2\tau +1}
\right )^{\frac {2\tau +1}{2}\cdot \alpha}t_2^{-2\alpha}\alpha !^{1-\tau}.
\end{equation*}
Hence Lemma \ref{Lemma:HermiteGroch} and
\eqref{Eq:CoeffF0FComp} give
$$
|c(F,\alpha )|
\lesssim 
\left (
\frac {2r_1}{2\tau +1}
\right )^{\frac {2\tau +1}{2}\cdot \alpha}\alpha !^{-\tau}
$$
and Lemma \ref{Lemma:Omega34Equiv} again implies that
\eqref{Eq:MainResult3Eq1} holds with $r_2$ in place of $r_0$.
This gives the result.
\end{proof}

\par

\begin{proof}[Proof of Theorems
\ref{Thm:MainResult3}--\ref{Thm:MainResult5}]
First suppose $\sigma \in (\frac 12,1)$, and let $\sigma _0
=\frac \sigma{2\sigma -1}$ and $\tau =\frac 1{2\sigma}$. Then
$$
\frac {2\sigma}{\sigma +1} = \frac {2}{2\tau +1}
\quad \text{and}\quad
\frac {2\sigma _0}{\sigma _0 -1} = \frac {2}{2\tau -1}.
$$
Theorem \ref{Thm:MainResult3} now follows from these observations
and Proposition \ref{Prop:MainResult345} in the case $\tau \in (\frac 12,1)$.

\par

In the same way, Theorem \ref{Thm:MainResult4} follows by
choosing $\tau =1$ in Proposition \ref{Prop:MainResult345}.

\par

Finally, suppose $\sigma \in (0,\frac 12)$, and let $\sigma _0
=\frac \sigma{1-2\sigma}$ and $\tau =\frac 1{2\sigma}$. Then
$$
\frac {2\sigma}{\sigma +1} = \frac {2}{2\tau +1}
\quad \text{and}\quad
\frac {2\sigma _0}{\sigma _0 +1} = \frac {2}{2\tau -1},
$$
and Theorem \ref{Thm:MainResult3} follows from these observations
and Proposition \ref{Prop:MainResult345} in the case $\tau >1$.
\end{proof}

\par

Next we prove Propositions \ref{Prop:MainResult6} and
\ref{Prop:MainResult6No2} and thereby Theorem
\ref{Thm:MainResult6}.

\par

\begin{proof}[Proof of Propositions \ref{Prop:MainResult6} and
\ref{Prop:MainResult6No2}]
Let $r,r_j,R_j\in \rr d_+$, $j=0,1,2,3$, be such that 
$$
r_0<r_1<r_2<r_3<r\quad \text{and}\quad
R_j=s
\left (
\frac {1-2s}{r_j}
\right )
^{\frac {1-2s}{2s}}.
$$
First suppose that $F\in A(\cc d)$ satisfies \eqref{Eq:MainResult6A}
and let $F_0$ be the formal power series expansion with coefficients
given by \eqref{Eq:CoeffF0FComp}. Then $F=\Pi _A(F_0\cdot \chi )$.

\par

By Lemma \ref{Lemma:MainResult6A} we get
$$
|c(F,\alpha )|\lesssim e^{-(R_{1,1}|\alpha _1|^{\frac 1{2s}}
+\cdots + R_{1,d}|\alpha _d|^{\frac 1{2s}})}.
$$
Hence Lemma \ref{Lemma:HermiteGroch}
and \eqref{Eq:CoeffF0FComp} give
\begin{multline*}
|c(F_0,\alpha )| \lesssim \alpha ! t_1^{-2\alpha}
e^{-(R_{1,1}|\alpha _1|^{\frac 1{2s}}
+\cdots + R_{1,d}|\alpha _d|^{\frac 1{2s}})}
\\[1ex]
\lesssim
e^{-(R_{2,1}|\alpha _1|^{\frac 1{2s}}
+\cdots + R_{2,d}|\alpha _d|^{\frac 1{2s}})}.
\end{multline*}
In the last inequality we have used the fact that $s<\frac 12$,
which implies that $R_2<R_1$ and
$$
t_1^{-2\alpha}\alpha ! \lesssim 
e^{(R_{1,1}-R_{2,1})|\alpha _1|^{\frac 1{2s}}
+\cdots + (R_{1,d}-R_{2,d})|\alpha _d|^{\frac 1{2s}}}.
$$

By applying Lemma \ref{Lemma:MainResult6A} again it
follows that $F_0$ satisfies \eqref{Eq:MainResult6B} with
$r_3$ in place of $r_0$. This gives (1).

\par

Suppose instead that $F_0\in A(\cc d)$ satisfies \eqref{Eq:MainResult6B}
and let $F=\Pi _A(F_0\cdot \chi )$. Then Lemma
\ref{Lemma:MainResult6A} gives
$$
|c(F_0,\alpha )|
\lesssim
e^{-(R_{1,1}|\alpha _1|^{\frac 1{2s}}
+\cdots + R_{1,d}|\alpha _d|^{\frac 1{2s}})},
$$
and it follows from
Lemma \ref{Lemma:HermiteGroch} and
\eqref{Eq:CoeffF0FComp} that
\begin{multline*}
|c(F,\alpha )| \lesssim \alpha ! ^{-1}t_2^{2\alpha}
e^{-(R_{1,1}|\alpha _1|^{\frac 1{2s}}
+\cdots + R_{1,d}|\alpha _d|^{\frac 1{2s}})}
\\[1ex]
\lesssim
e^{-(R_{1,1}|\alpha _1|^{\frac 1{2s}}
+\cdots + R_{1,d}|\alpha _d|^{\frac 1{2s}})}.
\end{multline*}
By applying Lemma \ref{Lemma:MainResult6A} again we deduce
\eqref{Eq:MainResult6A} with $r_2$ in place of $r_0$, and Proposition
\ref{Prop:MainResult6} follows.

\par

Proposition \ref{Prop:MainResult6No2} is a straight-forward
consequence of \eqref{Eq:CoeffF0FComp}. The details are left
for the reader.
\end{proof}

\par

\section{Characterizations of Pilipovi{\'c} spaces}\label{sec3}

\par

In this section we combine Lemma \ref{GrochSpaceBargm} with
Theorems \ref{Thm:MainResult1}--\ref{Thm:MainResult6} to
get characterizations of Pilipovi{\'c} spaces.

\par

We begin with the following characterization of
$\maclH _{\flat _1}$. The result is a straight-forward
combination of Lemma \ref{GrochSpaceBargm} and
Theorem \ref{Thm:MainResult1}. The details are left
for the reader.

\par

\begin{prop}\label{Prop:MainResult1Conseq}
Let $\phi$ be as in \eqref{phidef}, $r\in \rr d_+$, $\chi _r$
be the characteristic function for
$\overline{D_r(0)}$ and let $s>0$. Then the following conditions
are equivalent:
\begin{enumerate}
\item $f\in \maclH _{\flat _1}(\rr d)$;

\vrum

\item $f=V_\phi ^*F$ for some $F\in \maclE _s'(\rr {2d})$;

\vrum

\item $f=V_\phi ^*F$ for some $F\in \mascE '(\rr {2d})
\cap L^\infty (\rr {2d})$;

\vrum

\item $f=V_\phi ^*F$ for some $F\in \mascE '(\rr {2d})
\cap L^\infty (\rr {2d})$ which satisfies
\begin{equation}\label{Eq:F0toF}
F(x,\xi )= F_0(x-i\xi )e^{-i\scal x\xi}\chi (x,\xi )
\end{equation}
for some $r\in \rr d_+$, $\chi =\chi _r$ and
$F_0\in A(\overline {D_r(0)})$.
\end{enumerate}
\end{prop}

\par

\begin{rem}
It is clear that $\chi$ in Proposition
\ref{Prop:MainResult1Conseq} can be chosen as any
$\chi \in \maclR ^\infty (\cc d)$ with suitable
support properties.
\end{rem}

\par

\begin{rem}
In \eqref{Eq:FtoF0} there is a factor
$e^{\frac 12(|x|^2+|\xi |^2)}$ which is absent in
\eqref{Eq:F0toF}. We notice that this factor is not needed
in \eqref{Eq:F0toF} because $\maclR ^\infty _{t_1,t_2}(\cc d)$
is invariant under multiplications of such functions.
\end{rem}

\par

The next result follows from Lemma \ref{GrochSpaceBargm}
and Theorems \ref{Thm:MainResult2}--\ref{Thm:MainResult6}.
The details are left for the reader.

\par

\begin{prop}\label{Prop:MainResult2Conseq}
Let $\phi$ be as in \eqref{phidef} and $\chi \in
\maclR ^\infty (\cc d)$. Then the following conditions
are equivalent:
\begin{enumerate}
\item $f\in \maclH _{0,\flat _1}(\rr d)$;

\vrum

\item $f=V_\phi ^*F$ for some $F\in \mascE '(\rr {2d})
\cap L^\infty (\rr {2d})$ which satisfies \eqref{Eq:F0toF}
for some $F_0\in A(\cc d)$.
\end{enumerate}
\end{prop}

\par

\begin{prop}\label{Prop:MainResult3Conseq}
Let $\phi$ be as in \eqref{phidef}, $\chi \in
\maclR ^\infty (\cc d)$, $\sigma \in (\frac 12,1)$ and let
$$
\sigma _0 = \frac \sigma{2\sigma -1}.
$$
Then the following conditions
are equivalent:
\begin{enumerate}
\item $f\in \maclH _{\flat _\sigma}(\rr d)$ ($f\in \maclH
_{0,\flat _\sigma}(\rr d)$);

\vrum

\item $f=V_\phi ^*F$ for some $F\in \mascE '(\rr {2d})
\cap L^\infty (\rr {2d})$ which satisfies \eqref{Eq:F0toF}
for some $F_0\in \maclA _{0,\flat _{\sigma _0}}'(\cc d)$
($F_0\in \maclA _{\flat _{\sigma _0}}'(\cc d)$).
\end{enumerate}
\end{prop}

\par

\begin{prop}\label{Prop:MainResult4Conseq}
Let $\phi$ be as in \eqref{phidef}, $\chi \in
\maclR ^\infty (\cc d)$ and let $\sigma =\frac 12$.
Then the following conditions are equivalent:
\begin{enumerate}
\item $f\in \maclH _{\flat _{\sigma}}(\rr d)$ ($f\in \maclH
_{0,\flat _\sigma}(\rr d)$);

\vrum

\item $f=V_\phi ^*F$ for some $F\in \mascE '(\rr {2d})
\cap L^\infty (\rr {2d})$ which satisfies \eqref{Eq:F0toF}
for some $F_0\in \maclA _{0,1/2}'(\cc d)$
($F_0\in \maclA _{0,1/2}(\cc d)$).
\end{enumerate}
\end{prop}

\par

\begin{prop}\label{Prop:MainResult5Conseq}
Let $\phi$ be as in \eqref{phidef}, $\chi \in
\maclR ^\infty (\cc d)$, $\sigma \in (0,\frac 12)$ and let
$$
\sigma _0 = \frac \sigma{1-2\sigma}.
$$
Then the following conditions
are equivalent:
\begin{enumerate}
\item $f\in \maclH _{\flat _\sigma}(\rr d)$ ($f\in \maclH
_{0,\flat _\sigma}(\rr d)$);

\vrum

\item $f=V_\phi ^*F$ for some $F\in \mascE '(\rr {2d})
\cap L^\infty (\rr {2d})$ which satisfies \eqref{Eq:F0toF}
for some $F_0\in \maclA _{\flat _{\sigma _0}}(\cc d)$
($F_0\in \maclA _{0,\flat _{\sigma _0}}(\cc d)$).
\end{enumerate}
\end{prop}

\par

\begin{prop}\label{Prop:MainResult6Conseq}
Let $\phi$ be as in \eqref{phidef}, $\chi \in
\maclR ^\infty (\cc d)$ and let $s \in (0,\frac 12)$.
Then the following conditions
are equivalent:
\begin{enumerate}
\item $f\in \maclH _{s}(\rr d)$ ($f\in \maclH
_{s}(\rr d)$);

\vrum

\item $f=V_\phi ^*F$ for some $F\in \mascE '(\rr {2d})
\cap L^\infty (\rr {2d})$ which satisfies \eqref{Eq:F0toF}
for some $F_0\in \maclA _{s}(\cc d)$
($F_0\in \maclA _{0,s}(\cc d)$).
\end{enumerate}
\end{prop}

\par

%\end{savenotes}

\end{document}